\documentclass[english]{svjour3}
\usepackage[T1]{fontenc}
\usepackage[latin9]{inputenc}
\usepackage[a4paper]{geometry}
\usepackage{color}
\usepackage{babel}
\usepackage{url}
\usepackage{amsmath}
\usepackage{amssymb}
\usepackage{graphicx}
\usepackage{esint}
\usepackage[unicode=true,
 bookmarks=true,bookmarksnumbered=true,bookmarksopen=false,
 breaklinks=true,pdfborder={0 0 1},backref=false,colorlinks=false]
 {hyperref}
\hypersetup{pdftitle={Maxwell Time and Conjgate Loci},
 pdfauthor={Yasir Awais Butt}}
\usepackage{breakurl}

\makeatletter
\numberwithin{equation}{section}

\usepackage{authblk}
\smartqed  % flush right qed marks, e.g. at end of proof
\numberwithin{equation}{section}
\numberwithin{theorem}{section}
\numberwithin{proposition}{section}
\makeatother

\makeatother

\begin{document}

\title{Maxwell Strata and Conjugate Points in the Sub-Riemannian Problem
on the Lie Group SH(2)}

\author{Yasir Awais Butt, Yuri L. Sachkov, Aamer Iqbal Bhatti}

\institute{Yasir Awais Butt \at Department of Electronic Engineering\\
Muhammad Ali Jinnah University\\
Islamabad, Pakistan\\
Tel.: +92-51-111878787\\
\email{yasir\_awais2000@yahoo.com}\\
\and Yuri L. Sachkov \at Program Systems Institute\\
Pereslavl-Zalessky, Russia\\
\email{yusachkov@gmail.com}\\
\and Aamer Iqbal Bhatti\at Department of Electronic Engineering\\
Muhammad Ali Jinnah University\\
Islamabad, Pakistan\\
Tel.: +92-51-111878787\\
\email{aib@jinnah.edu.pk}}
\maketitle
\begin{abstract}
We study local and global optimality of geodesics in the left invariant
sub-Riemannian problem on the Lie group $\mathrm{SH}(2)$. We obtain
the complete description of the Maxwell points corresponding to the
discrete symmetries of the vertical subsystem of the Hamiltonian system.
An effective upper bound on the cut time is obtained in terms of the
first Maxwell times. We study the local optimality of extremal trajectories
and prove the lower and upper bounds on the first conjugate times.
We also obtain the generic time interval for the $n$-th conjugate
time which is important in the study of sub-Riemannian wavefront.
Based on our results of $n$-th conjugate time and $n$-th Maxwell
time, we prove a generalization of Rolle's theorem that between any
two consecutive Maxwell points, there is exactly one conjugate point
along any geodesic.

\keywords{Sub-Riemannian geometry, Special hyperbolic group SH(2), Maxwell points,
Cut time, Conjugate time} 

\subclass{49J15, 93B27, 93C10, 53C17, 22E30}
\end{abstract}

\section{Introduction}

Geometric control theory for linear systems was initiated in 1970s
\cite{Marro} and was extended to nonlinear systems in 1980s \cite{Isidori}.
An important class of problems addressed by geometric control theory
consists of control of the dynamical systems subjected to nonholonomic
constraints \cite{Arnold_Non_Hol}, \cite{agrachev_sachkov}, \cite{Bloch}.
It turns out that the optimal control of a large number of these physically
interesting systems reduces to finding geodesics with respect to a
sub-Riemannian metric \cite{Bloch}. Owing to the motivations and
ramifications of sub-Riemannian problems in control theory, research
on the sub-Riemannian problem on the group of motions of pseudo Euclidean
plane was initiated in \cite{Extremal_Pseudo_Euclid}. Motions of
the pseudo Euclidean plane form the Lie group $\mathrm{SH}(2)$ \cite{Ja.Vilenkin}.
The sub-Riemannian problem on $\mathrm{SH}(2)$ seeks to obtain optimal
control for the system that comprises left invariant vector fields
with 2-dimensional linear control input and energy cost functional.
The study of sub-Riemannian problem on $\mathrm{SH}(2)$ bears significance
in the program of complete study of all the left-invariant sub-Riemannian
problems on 3-dimensional Lie groups following the classification
in terms of the basic differential invariants \cite{agrachev_barilari}.
The Lie group $\mathrm{SH}(2)$ gives one of the Thurston's 3-dimensional
geometries called Sol \cite{Thurston}.

In \cite{Extremal_Pseudo_Euclid} parametrization of extremal trajectories
in the sub-Riemannian problem on the Lie group SH(2) was obtained
via application of Pontryagin Maximum Principle (PMP). Since PMP provides
only the necessary conditions for the optimal trajectories, the optimality
conditions for given boundary points are therefore satisfied by a
countable number of competing curves with different integral cost,
not because of the optimality, but because the curves terminate on
the boundary of the extended attainable set \cite{Gamkrelidze}. Second
order and global optimality conditions such as conjugate points and
Maxwell points are therefore investigated to establish optimality. 

This paper is an extension of \cite{Extremal_Pseudo_Euclid} in which
we obtained complete parametrization of extremal trajectories $q_{t}=(x_{t},y_{t},z_{t})$
and stated the general conditions for existence of the Maxwell points
in terms of the equations $R_{i}(q_{t})=0$ and $z_{t}=0$ (the functions
$R_{i}$ are given below in (\ref{eq:2.16})). We now extend our analysis
to completely characterize the Maxwell points and obtain the first
Maxwell times. The first Maxwell time forms an effective upper bound
on the cut time. We then investigate local optimality of the geodesics
via description of conjugate points. The roots of the Jacobian of
the exponential mapping are studied, and lower and upper bounds on
the first conjugate time as well as the $n$-th conjugate time are
obtained. We show that the function that gives the upper bound on
the cut time provides the lower bound of the first conjugate time. 

The paper is organized as follows. Section 2 presents a brief review
of our results from \cite{Extremal_Pseudo_Euclid}. In Section 3,
we describe the roots of the functions $R_{i}(q_{t})$ and $z_{t}$.
These roots allow us to calculate the first Maxwell times and an effective
upper bound on the cut time. Section 4 pertains to the local optimality
analysis of geodesics via description of conjugate points. We compute
the lower and upper bound on the first conjugate time as well as the
bounds on the $n$-th conjugate time. Section 4 ends with the 3-dimensional
plots of sub-Riemannian wavefront and sub-Riemannian spheres. Sections
5 and 6 pertain to future work and conclusion respectively.

\section{Previous Work}

\subsection{Problem Statement\label{sec:Problem-Statement}}

Motions of the pseudo Euclidean plane are distance and orientation
preserving maps of the hyperbolic plane. These motions describe hyperbolic
roto-translations of the pseudo Euclidean plane and form a 3-dimensional
Lie group known as the special hyperbolic group $\mathrm{SH(2)}$
\cite{Ja.Vilenkin}. The sub-Riemannian problem on the Lie group $\mathrm{SH(2)}$
reads as follows \cite{Extremal_Pseudo_Euclid}:
\begin{eqnarray}
\dot{x} & = & u_{1}\cosh z,\quad\dot{y}=u_{1}\sinh z,\quad\dot{z}=u_{2},\label{eq:2.1}\\
q & = & (x,y,z)\in M=\mathrm{SH(2)\cong\mathbb{R}^{3}},\quad x,y,z\in\mathbb{R},\quad(u_{1},u_{2})\in\mathbb{R}^{2},\label{eq:2.2}\\
q(0) & = & (0,0,0),\qquad q(t_{1})=q_{1}=(x_{1},y_{1},z_{1}),\label{eq:2.3}\\
l & = & \int_{0}^{t_{1}}\sqrt{u_{1}^{2}+u_{2}^{2}}\, dt\to\min.\label{eq:2.4}
\end{eqnarray}
By Cauchy-Schwarz inequality, the sub-Riemannian length functional
$l$ minimization problem (\ref{eq:2.4}) is equivalent to the problem
of minimizing the following action functional with fixed $t_{1}$
\cite{sachkov_lectures}:
\begin{equation}
J=\frac{1}{2}\intop_{0}^{t_{1}}(u_{1}^{2}+u_{2}^{2})dt\rightarrow\min.\label{eq:2.5}
\end{equation}

\subsection{Known Results\label{sec:Previous-Work}}

We now briefly review the results from \cite{Extremal_Pseudo_Euclid}
as a ready reference in this paper. System (\ref{eq:2.1}) satisfies
the bracket generating condition and is hence globally controllable
\cite{Chow},\cite{Ravchevsky}. Existence of optimal trajectories
for the optimal control problem (\ref{eq:2.1})--(\ref{eq:2.5}) follows
from Filippov\textquoteright s theorem \cite{agrachev_sachkov}. Since
the problem is 3-dimensional contact, it is well known that abnormal
extremal trajectories are constant \cite{Agrachev_Exp_Map}. We applied
PMP \cite{agrachev_sachkov} to (\ref{eq:2.1})--(\ref{eq:2.5}) to
derive the normal Hamiltonian system. It turns out that the vertical
part of the normal Hamiltonian system is a double covering of a mathematical
pendulum. The normal Hamiltonian system is given as:
\begin{eqnarray}
\dot{\gamma} & = & c,\quad\dot{c}=-\sin\gamma,\quad\lambda=(\gamma,c)\in C\cong(2S_{\gamma}^{1})\times\mathbb{R}_{c},\quad2S_{\gamma}^{1}=\mathbb{R}/(4\pi\mathbb{Z}),\label{eq:2.6}\\
\dot{x} & = & \cos\frac{\gamma}{2}\cosh z,\quad\dot{y}=\cos\frac{\gamma}{2}\sinh z,\quad\dot{z}=\sin\frac{\gamma}{2}.\label{eq:2.7}
\end{eqnarray}
The initial cylinder of the vertical subsystem was decomposed into
the following subsets based upon the pendulum energy that correspond
to various pendulum trajectories:
\begin{eqnarray*}
C & = & \bigcup_{i=1}^{5}C_{i},
\end{eqnarray*}
where
\begin{eqnarray}
C_{1} & = & \left\{ \lambda\in C\,\vert\, E\in(-1,1)\right\} ,\label{eq:2.8}\\
C_{2} & = & \left\{ \lambda\in C\,\vert\, E\in(1,\infty)\right\} ,\\
C_{3} & = & \left\{ \lambda\in C\,\vert\, E=1,c\neq0\right\} ,\\
C_{4} & = & \left\{ \lambda\in C\,\vert\, E=-1,\, c=0\right\} =\left\{ (\gamma,c)\in C\,\vert\,\gamma=2\pi n,\, c=0\right\} ,\quad n\in\mathbb{N},\\
C_{5} & = & \left\{ \lambda\in C\,\vert\, E=1,\, c=0\right\} =\left\{ (\gamma,c)\in C\,\vert\,\gamma=2\pi n+\pi,\, c=0\right\} ,\quad n\in\mathbb{N}.\label{eq:2.12}
\end{eqnarray}
\begin{figure}
\centering{}\includegraphics[scale=0.5]{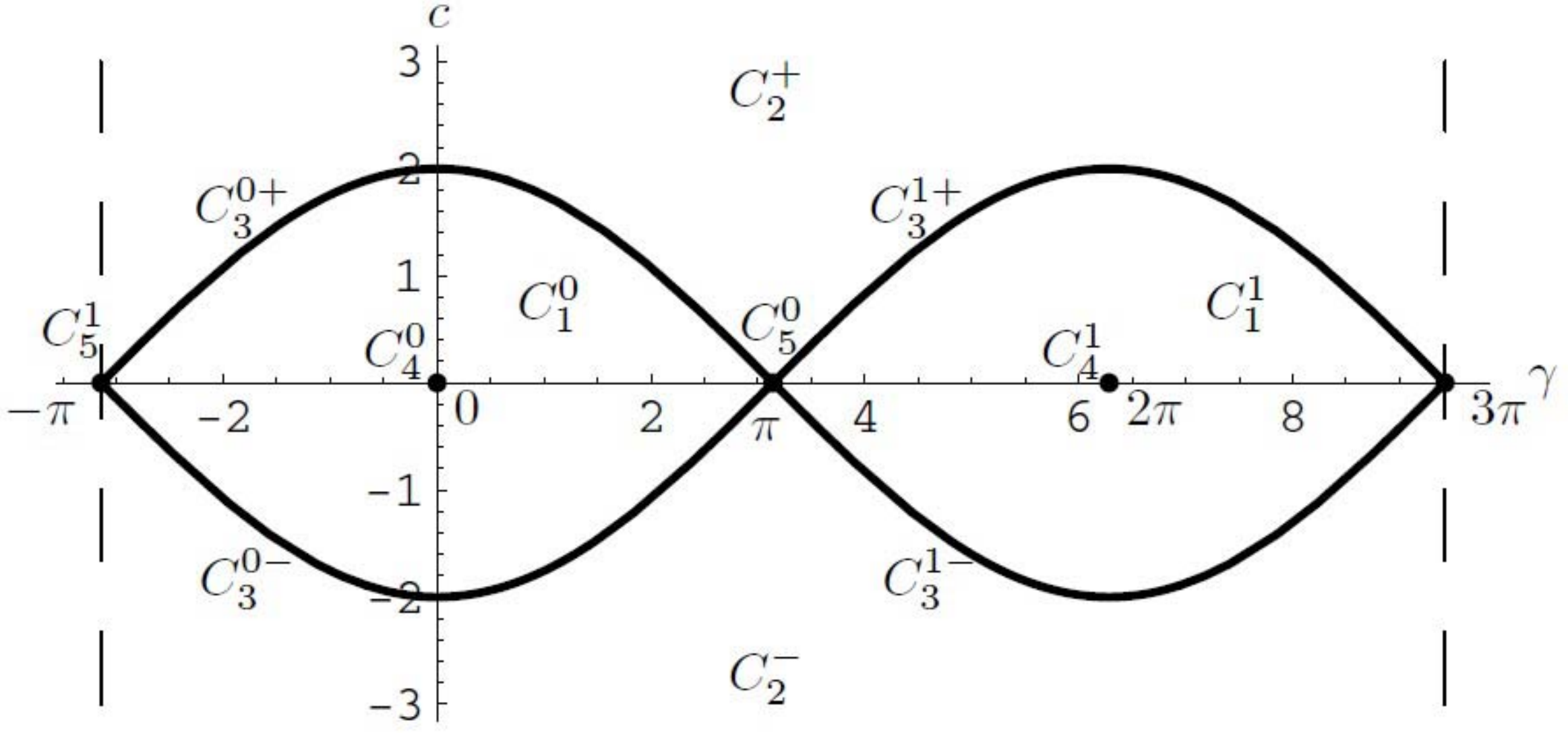}\protect\caption{\label{fig:Decomposition}Decomposition of the Phase Cylinder $C$
of the Pendulum}
\end{figure}
We defined elliptic coordinates $(\varphi,k)$ for $\lambda\in\cup_{i=1}^{3}C_{i}\subset C$
and proved that the flow of the pendulum is rectified in these coordinates.
Note that $k$ was defined as the reparametrized energy and $\varphi$
was defined as the reparametrized time of motion of the pendulum \cite{Extremal_Pseudo_Euclid}.
Integration of the horizontal subsystem in elliptic coordinates follows
from integration of the vertical subsystem and the resulting extremal
trajectories are parametrized by the Jacobi elliptic functions $\mathrm{sn}(\varphi,k)$,
$\mathrm{cn}(\varphi,k)$, $\mathrm{dn}(\varphi,k)$, $\mathrm{E}(\varphi,k)=\intop_{0}^{\varphi}\mathrm{dn^{2}}(t,k)dt$
(Theorems 5.1--5.5 \cite{Extremal_Pseudo_Euclid}). The results of
integration for $\lambda\in C_{i},\quad i=1,2,3$, are summarized
as: 
\begin{itemize}
\item Case 1 : $\lambda=(\varphi,k)\in C_{1}$
\begin{equation}
\left(\begin{array}{c}
x_{t}\\
y_{t}\\
z_{t}
\end{array}\right)=\left(\begin{array}{c}
\frac{s_{1}}{2}\left[\left(w+\frac{1}{w\left(1-k^{2}\right)}\right)\left[\mathrm{E}(\varphi_{t})-\mathrm{E}(\varphi)\right]+\left(\frac{k}{w(1-k^{2})}-kw\right)\left[\mathrm{sn}\,\varphi_{t}-\mathrm{sn}\,\varphi\right]\right]\\
\frac{1}{2}\left[\left(w-\frac{1}{w\left(1-k^{2}\right)}\right)\left[\mathrm{E}(\varphi_{t})-\mathrm{E}(\varphi)\right]-\left(\frac{k}{w\left(1-k^{2}\right)}+kw\right)\left[\mathrm{sn}\,\varphi_{t}-\mathrm{sn}\,\varphi\right]\right]\\
s_{1}\ln\left[(\mathrm{dn}\,\varphi_{t}-k\mathrm{cn}\,\varphi_{t}).w\right]
\end{array}\right),\label{eq:2.13}
\end{equation}
where $w=\frac{1}{\mathrm{dn}\varphi-k\mathrm{cn}\varphi}$, $s_{1}=\mathrm{sgn}\left(\cos\frac{\gamma}{2}\right)$
and $\varphi_{t}=\varphi+t$.
\item Case 2 : $\lambda=(\psi,k)\in C_{2}$ 
\begin{eqnarray}
x_{t} & = & \frac{1}{2}\left(\frac{1}{w(1-k^{2})}-w\right)\left[\mathrm{E}(\psi_{t})-\mathrm{E}(\psi)-k^{\prime2}\left(\psi_{t}-\psi\right)\right]\nonumber \\
 & + & \frac{1}{2}\left(kw+\frac{k}{w(1-k^{2})}\right)\left[\mathrm{sn}\,\psi_{t}-\mathrm{sn}\,\psi\right],\nonumber \\
y_{t} & = & -\frac{s_{2}}{2}\left(\frac{1}{w(1-k^{2})}+w\right)\left[\mathrm{E}(\psi_{t})-\mathrm{E}(\psi)-k^{\prime2}(\psi_{t}-\psi)\right]\nonumber \\
 & + & \frac{s_{2}}{2}\left(kw-\frac{k}{w(1-k^{2})}\right)\left[\mathrm{sn}\,\psi_{t}-\mathrm{sn}\,\psi\right],\nonumber \\
z_{t} & = & s_{2}\ln\left[\left(\mathrm{dn}\,\psi_{t}-k\mathrm{cn}\,\psi_{t}\right).w\right],\label{eq:2.14}
\end{eqnarray}
where $\psi=\frac{\varphi}{k}$, $\quad\psi_{t}=\frac{\varphi_{t}}{k}=\psi+\frac{t}{k}$
and $w=\frac{1}{\mathrm{dn}\,\psi-k\mathrm{cn}\,\psi}$, $s_{2}=\mathrm{sgn}\, c$. 
\item Case 3 : $\lambda=(\varphi,k)\in C_{3}$
\begin{equation}
\left(\begin{array}{c}
x_{t}\\
y_{t}\\
z_{t}
\end{array}\right)=\left(\begin{array}{c}
\frac{s_{1}}{2}\left[\frac{1}{w}\left(\varphi_{t}-\varphi\right)+w\left(\tanh\varphi_{t}-\tanh\varphi\right)\right]\\
\frac{s_{2}}{2}\left[\frac{1}{w}\left(\varphi_{t}-\varphi\right)-w\left(\tanh\varphi_{t}-\tanh\varphi\right)\right]\\
-s_{1}s_{2}\ln[w\,\textrm{sech}\,\varphi_{t}]
\end{array}\right),\label{eq:2.15}
\end{equation}
where $w=\cosh\varphi$. 
\end{itemize}
The phase portrait of the pendulum admits a discrete group of symmetries
$G=\{Id,\varepsilon^{1},\ldots,\varepsilon^{7}\}$. The symmetries
$\varepsilon^{i}$ are reflections and translations about the coordinates
axes $(\gamma,c)$. These symmetries are exploited to state the general
conditions on Maxwell strata in terms of the functions $z_{t}$ and
$R_{i}$ given as: 
\begin{equation}
R_{1}=y\cosh\frac{z}{2}-x\sinh\frac{z}{2},\quad R_{2}=x\cosh\frac{z}{2}-y\sinh\frac{z}{2}.\label{eq:2.16}
\end{equation}
We define the Maxwell sets $\mathrm{MAX}^{i},\quad i=1,\ldots,7$,
resulting from the reflections $\varepsilon^{i}$ of the extremals
in the preimage of the exponential mapping $N$ as:
\[
\mathrm{MAX}^{i}=\left\{ \text{\ensuremath{\nu}}=(\text{\ensuremath{\lambda}},t)\text{\ensuremath{\in}}N\quad|\quad\lambda\neq\lambda{}^{i},\quad\mathrm{Exp}(\lambda,t)=\mathrm{Exp}(\lambda^{i},t)\right\} ,
\]
where $\lambda=\varepsilon^{i}(\lambda).$ The corresponding Maxwell
strata in the image of the exponential mapping are defined as:

\[
\mathrm{Max}^{i}=\mathrm{Exp}(\mathrm{MAX}^{i})\subset M.
\]
 General description of the Maxwell strata is then given as:
\begin{flalign*}
(1)\qquad\nu\in\mathrm{MAX}^{1} & \Leftrightarrow\left\{ \begin{array}{ccc}
R_{1}(q)=0, & \mathrm{cn}\tau\neq0, & \textrm{ for}\,\lambda\in C_{1}\qquad\\
R_{1}(q)=0, &  & \textrm{for}\,\lambda\in C_{2}\cup C_{3}
\end{array}\right\} , & {}\\
(2)\qquad\nu\in\mathrm{MAX}^{2} & \Leftrightarrow\left\{ \begin{array}{ccc}
z=0, & \mathrm{sn}\tau\neq0,\quad\,\, & \textrm{for }\lambda\in C_{1}\cup C_{2}\\
z=0, & \tau\neq0,\quad\quad\, & \textrm{for }\lambda\in C_{3}\qquad
\end{array}\right\} , & {}\\
(3)\qquad\nu\in\mathrm{MAX}^{6} & \Leftrightarrow\left\{ \begin{array}{ccc}
R_{2}(q)=0, & \mathrm{cn}\tau\neq0, & \textrm{for }\lambda\in C_{2}\qquad\\
R_{2}(q)=0, &  & \textrm{for }\lambda\in C_{1}\cup C_{3}
\end{array}\right\} , & {}
\end{flalign*}
where 
\begin{eqnarray}
\tau & = & \frac{1}{2}\left(\varphi_{t}+\varphi\right),\, p=\frac{t}{2}\textrm{ when }\nu=(\lambda,t)\in N_{1}\cup N_{3},\label{eq:2.17}\\
\tau & = & \frac{1}{2k}\left(\varphi_{t}+\varphi\right),\, p=\frac{t}{2k}\textrm{ when }\nu=(\lambda,t)\in N_{2}.\label{eq:2.18}
\end{eqnarray}

\section{Complete Description of the Maxwell Strata}

\subsection{Roots of Equations $R_{i}(q_{t})=0$ and $z_{t}=0$}

We now study roots of the equations $R_{i}(q_{t})=0$ and $z_{t}=0$
to describe the Maxwell strata in the sub-Riemannian problem on $\mathrm{SH}(2)$.
The idea is to obtain a parametrization of the roots in terms of $\tau$
and $p$ defined in (\ref{eq:2.17})--(\ref{eq:2.18}). Using the
addition formulas for Jacobi elliptic functions we get the following
representation of the functions along the extremal trajectories:

\textbf{Case 1 - $\lambda\in C_{1}$:}
\begin{align}
\varphi_{t} & =\tau+p,\quad\varphi=\tau-p,\label{eq:3.1}\\
\sinh z_{t} & =s_{1}\frac{2k\,\mathrm{sn}p\,\mathrm{sn}\tau}{\Delta},\label{eq:3.2}\\
\sinh\frac{z_{t}}{2} & =s_{1}\frac{k\,\mathrm{sn}p\,\mathrm{sn}\tau}{\sqrt{\Delta}},\label{eq:3.3}\\
\cosh\frac{z_{t}}{2} & =\frac{1}{\sqrt{\Delta}},\label{eq:3.4}\\
R_{1}(q_{t}) & =\frac{2k}{1-k^{2}}\mathrm{cn}\tau\, f_{1}(p),\label{eq:3.5}\\
R_{2}(q_{t}) & =\frac{2s_{1}}{1-k^{2}}\mathrm{dn}\tau\, f_{2}(p),\label{eq:3.6}
\end{align}
where $\Delta=1-k^{2}\mathrm{sn}^{2}p\,\mathrm{sn}^{2}\tau$, $f_{1}(p)=\mathrm{cn}p\,\mathrm{E}(p)-\mathrm{sn}p\,\mathrm{dn}p$
and $f_{2}(p)=\mathrm{dn}p\,\mathrm{E}(p)-k^{2}\mathrm{sn}p\,\mathrm{cn}p$.

\textbf{Case 2 - $\lambda\in C_{2}$:}
\begin{align}
\frac{\varphi_{t}}{k} & =\tau+p,\quad\frac{\varphi}{k}=\tau-p,\label{eq:3.7}\\
\sinh z_{t} & =s_{2}\frac{2k\,\mathrm{sn}p\,\mathrm{sn}\tau}{\Delta},\label{eq:3.8}\\
\sinh\frac{z_{t}}{2} & =s_{2}\frac{k\,\mathrm{sn}p\,\mathrm{sn}\tau}{\sqrt{\Delta}},\label{eq:3.9}\\
\cosh\frac{z_{t}}{2} & =\frac{1}{\sqrt{\Delta}},\label{eq:3.10}\\
R_{1}(q_{t}) & =\frac{2s_{2}}{1-k^{2}}\mathrm{dn}\tau\, f_{3}(p),\label{eq:46}\\
R_{2}(q_{t}) & =\frac{2k}{1-k^{2}}\mathrm{cn}\tau\, f_{4}(p),\label{eq:47}
\end{align}
where $f_{3}(p)=-\mathrm{dn}p\mathrm{\, E}(p)+p\mathrm{\, dn}p(1-k^{2})+k^{2}\mathrm{sn}p\,\mathrm{cn}p$
and $f_{4}(p)=-\mathrm{cn}p\,\mathrm{E}(p)+p\,\mathrm{cn}p(1-k^{2})+\mathrm{sn}p\,\mathrm{dn}p$. 

\textbf{Case 3 - $\lambda\in C_{3}$:}
\begin{align}
\varphi_{t} & =\tau+p,\quad\varphi=\tau-p,\label{eq:3.13}\\
\sinh z & =2s_{1}s_{2}\frac{\sinh(\tau)\sinh(p)\cosh(\tau)\cosh(p)}{\Delta},\label{eq:3.14}\\
\sinh\frac{z_{t}}{2} & =s_{1}s_{2}\frac{\sinh(\tau)\sinh(p)}{\sqrt{\Delta}},\label{eq:3.15}\\
\cosh\frac{z_{t}}{2} & =\frac{\cosh(\tau)\cosh(p)}{\sqrt{\Delta}},\label{eq:3.16}\\
R_{1}(q_{t}) & =s_{2}\frac{2p-\sinh2p}{2\sqrt{\Delta}},\label{eq:3.17}\\
R_{2}(q_{t}) & =s_{1}\frac{2p+\sinh2p}{2\sqrt{\Delta}},\label{eq:3.18}
\end{align}
where $\Delta=\cosh^{2}\tau+\sinh^{2}p$.
\begin{proposition}
\label{prop:3.1}Let $t>0$.
\begin{flalign*}
(1)\qquad\mathrm{If}\,\lambda\in C_{1}\quad & \mathrm{then}\quad z_{t}=0\quad\Longleftrightarrow\quad p=2Kn,\quad\mathrm{sn}\tau=0.\\
(2)\qquad\mathrm{If}\,\lambda\in C_{2}\quad & \mathrm{then}\quad z_{t}=0\quad\Longleftrightarrow\quad p=2Kn,\quad\mathrm{sn}\tau=0.\\
(3)\qquad\mathrm{If}\,\lambda\in C_{3}\quad & \mathrm{then}\quad z_{t}=0\quad\Longleftrightarrow\quad p=0,\quad\tau=0. & {}
\end{flalign*}
\end{proposition}
\begin{proof}
Item (1) follows from (\ref{eq:3.2}), item (2) from (\ref{eq:3.8})
and item (3) from (\ref{eq:3.14}).\hfill$\square$\end{proof}

\begin{proposition}
\label{prop:f1p}The function $f_{1}(p)$ has an infinite number of
roots for any $k\in[0,1)$ given as:
\begin{align}
p & =p_{1}^{n}(k),\quad n\in\mathbb{Z},\label{eq:3.19}\\
p_{1}^{0} & =0,\label{eq:3.20}\\
p_{1}^{-n}(k) & =-p_{1}^{n}(k).\label{eq:3.21}
\end{align}
Moreover, the positive roots admit the bound:
\begin{equation}
p_{1}^{n}(k)\in\left(2nK\,,\,(2n+1)K\right),\quad n\in\mathbb{N},\quad k\in(0,1).\label{eq:3.22}
\end{equation}
\end{proposition}
\begin{proof}
Equalities (\ref{eq:3.20})--(\ref{eq:3.21}) follow directly from
the fact that $f_{1}(p)$ is odd.

To prove (\ref{eq:3.22}) consider the function $g_{1}(p)=f_{1}(p)/\mathrm{cn}p$,
which has the same roots as $f_{1}(p)$ and also:

\begin{align*}
\lim_{p\rightarrow(2n-1)K+}g_{1}(p) & \rightarrow+\infty,\\
\lim_{p\rightarrow(2n+1)K-}g_{1}(p) & \rightarrow-\infty,\\
g_{1}^{\prime}(p) & =-\frac{(1-k^{2})\mathrm{sn}^{2}p}{\mathrm{cn^{2}p}}\leq0.
\end{align*}
Hence $g_{1}(p)$ is decreasing on the interval $((2n-1)K\,,\,(2n+1)K)$
approaching $\pm\infty$ on the boundaries of the interval. It follows
that $g_{1}(p)$ and therefore $f_{1}(p)$ admit a unique root $p=p_{1}^{n}(k)$
in each interval $((2n-1)K\,,\,(2n+1)K)$. Since $g_{1}(2nK)>0,$
for $n\in\mathbb{N},$ therefore $p_{1}^{n}(k)\in(2nK,(2n+1)K).$
Plots of the functions $f_{1}(p)$ and $g_{1}(p)$ for $k=0.9$ are
given in Figure \ref{fig:f1g1}. \hfill$\square$
\begin{figure}
\centering{}\includegraphics[scale=0.5]{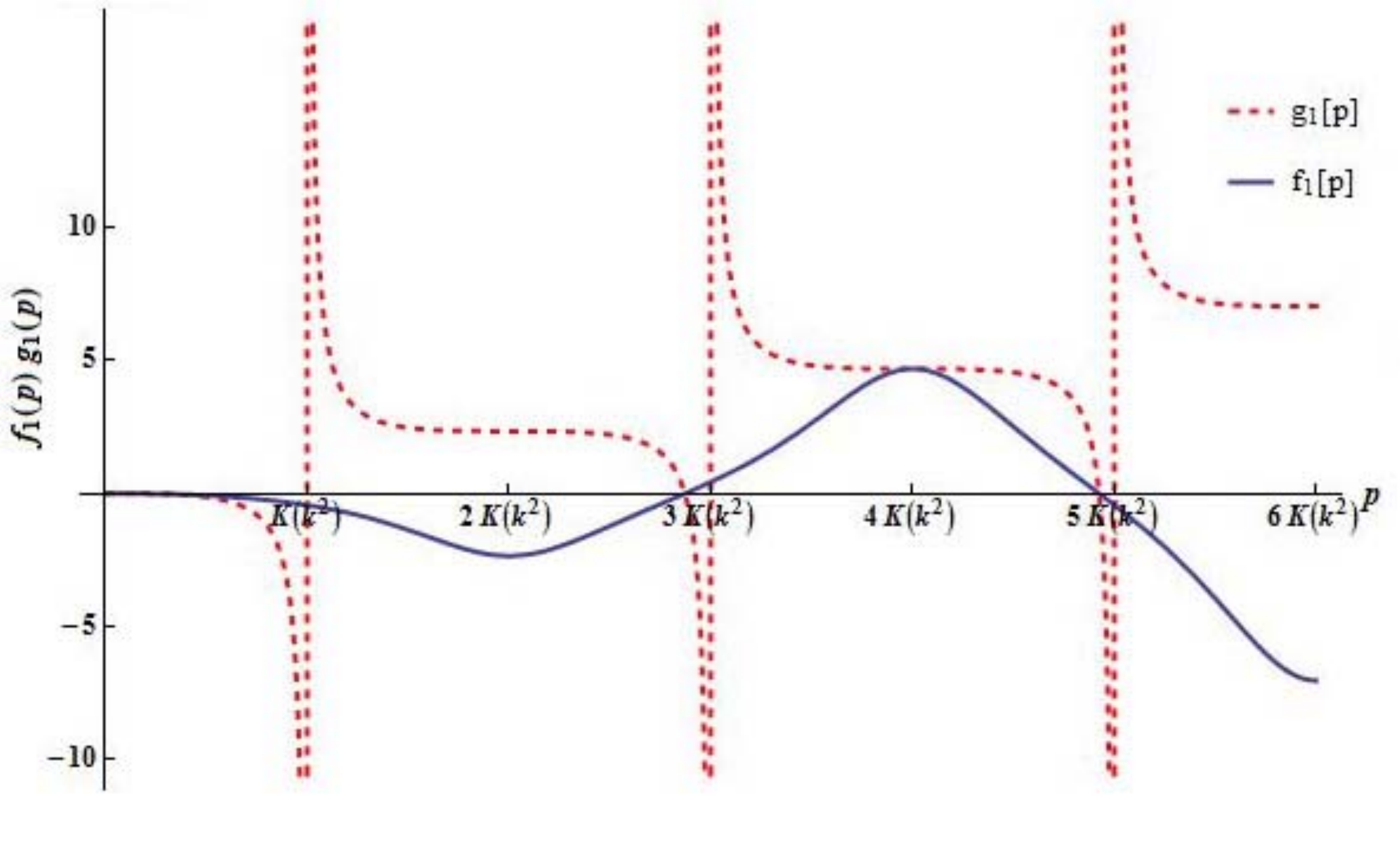}\protect\caption{\label{fig:f1g1}Roots of the functions $f_{1}(p)$ and $g_{1}(p)$}
\end{figure}
\end{proof}

\begin{lemma}
\label{lem:f2p}The function $f_{2}(p)$ is positive for any $p>0$
and $k\in(0,1).$\end{lemma}
\begin{proof}
Consider the function $g_{2}(p)=f_{2}(p)/\mathrm{dn}p$ where
\begin{align*}
g_{2}^{\prime}(p) & =\frac{1-k^{2}}{\mathrm{dn}^{2}p}>0.
\end{align*}
Since $g_{2}(0)=0$ therefore $g_{2}(p)>0$ and $f_{2}(p)>0$ for
$p>0$. \hfill$\square$
\end{proof}

\begin{lemma}
\label{lem:f3p}The function $f_{3}(p)$ is negative for any $p>0$
and $k\in(0,1).$\end{lemma}
\begin{proof}
Consider the function $g_{3}(p)=f_{3}(p)/\mathrm{dn}p$ which has
the same roots as $f_{3}(p)$ such that

\begin{align*}
g_{3}^{\prime}(p) & =-\frac{(1-k^{2})k^{2}\mathrm{sn}^{2}p}{\mathrm{dn}^{2}p}\leq0.
\end{align*}
Since $g_{3}(0)=0$ therefore $g_{3}(p)<0$ and $f_{2}(p)<0$ for
$p>0$. \hfill$\square$\end{proof}

\begin{proposition}
\label{prop:f4p}The function $f_{4}(p)$ has an infinite number of
roots for any $k\in[0,1)$ given as:

\begin{align}
p & =p_{2}^{n}(k),\qquad n\in\mathbb{Z},\label{eq:3.23}\\
p_{2}^{0} & =0,\label{eq:3.24}\\
p_{2}^{-n}(k) & =-p_{2}^{n}(k).\label{eq:3.25}
\end{align}
Moreover, the positive roots admit the bound:
\begin{equation}
p_{2}^{n}(k)\in\left(2nK\,,\,(2n+1)K\right),\quad n\in\mathbb{N},\quad k\in(0,1).\label{eq:3.26}
\end{equation}
\end{proposition}
\begin{proof}
Equalities (\ref{eq:3.24})--(\ref{eq:3.25}) follow directly from
the fact that $f_{4}(p)$ is odd.

To prove (\ref{eq:3.26}) consider the function $g_{4}(p)=f_{4}(p)/\mathrm{cn}p$
which has the same roots as $f_{4}(p)$ and also:

\begin{align*}
\lim_{p\rightarrow(2n-1)K+}g_{4}(p) & \rightarrow-\infty,\\
\lim_{p\rightarrow(2n+1)K-}g_{4}(p) & \rightarrow+\infty,\\
g_{4}^{\prime}(p) & =\frac{1-k^{2}}{\mathrm{cn^{2}p}}>0.
\end{align*}
Hence $g_{4}(p)$ is increasing on the interval $((2n-1)K\,,\,(2n+1)K)$
approaching $\mp\infty$ on the boundary of the interval. It follows
that $g_{4}(p)$ and therefore $f_{4}(p)$ admit a unique root $p_{2}^{n}(k)$
on each such interval. Following an argument similar to the one in
Proposition \ref{prop:f1p}, it follows that $p_{2}^{n}(k)\in(2nK\,,\,(2n+1)K)$.
Plots of the functions $f_{4}(p)$ and $g_{4}(p)$ for $k=0.9$ are
given in Figure \ref{fig:f4g4}.\hfill$\square$
\begin{figure}
\centering{}\includegraphics[scale=0.5]{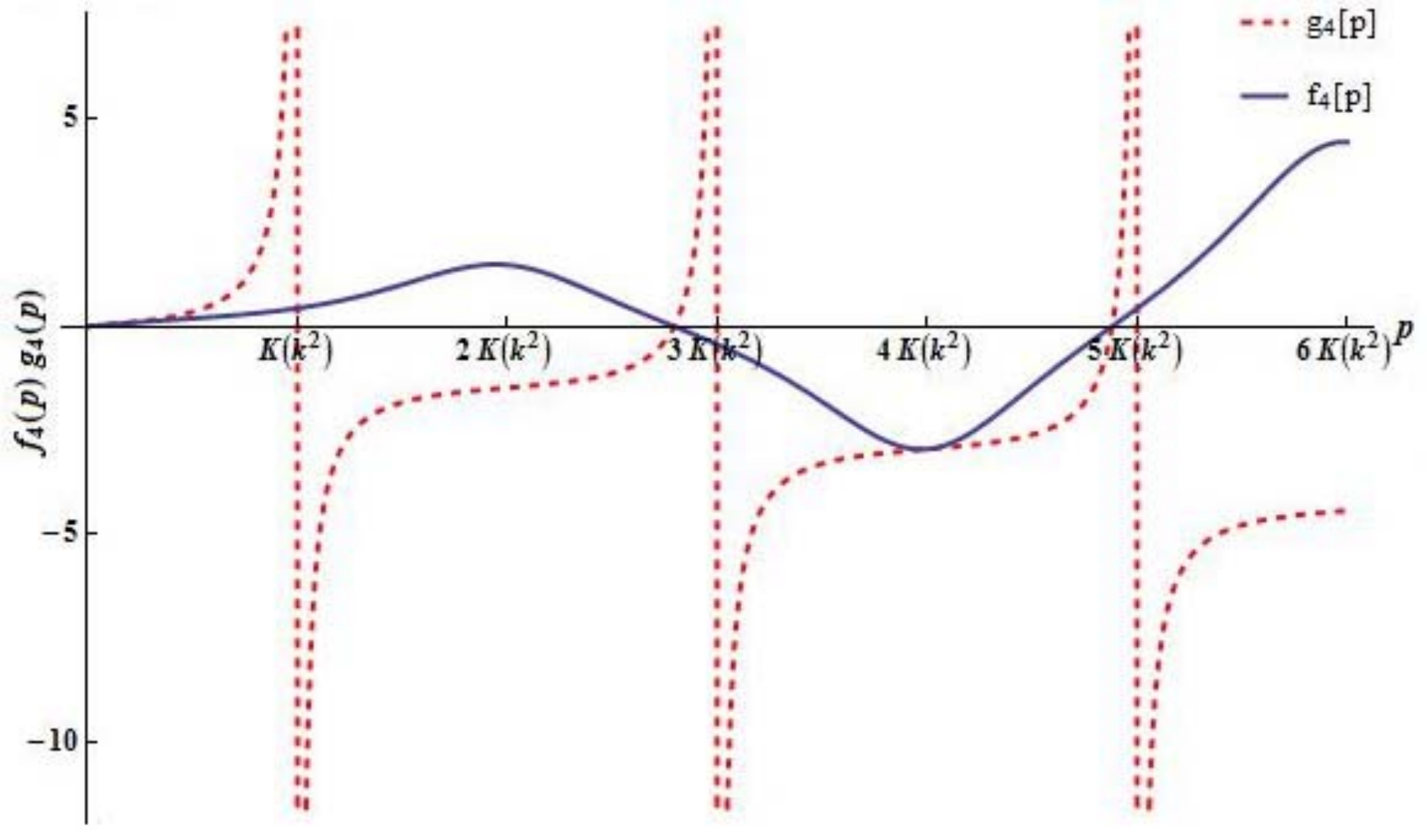}\protect\caption{\label{fig:f4g4}Roots of the functions $f_{4}(p)$ and $g_{4}(p)$}
\end{figure}
\end{proof}

\begin{proposition}
\label{prop:3.4}Let $t>0$.
\begin{flalign}
(1)\qquad\mathrm{If}\,\lambda\in C_{1}\quad & \mathrm{then}\quad R_{1}(q_{t})=0\quad\Longleftrightarrow\quad p=p_{1}^{n}(k)\textrm{ or }\mathrm{cn}\tau=0.\label{eq:55}\\
(2)\qquad\mathrm{If}\,\lambda\in C_{2}\quad & \mathrm{then}\quad R_{1}(q_{t})=0\quad\mathrm{is\, impossible}.\label{eq:56}\\
(3)\qquad\mathrm{If}\,\lambda\in C_{3}\quad & \mathrm{then}\quad R_{1}(q_{t})=0\quad\mathrm{is\, impossible}. & {}\label{eq:57}
\end{flalign}
\end{proposition}
\begin{proof}
Item (1) follows from (\ref{eq:3.5}) and Proposition \ref{prop:f1p}.
Item (2) is given from (\ref{eq:46}) and Lemma \ref{lem:f3p}. Item
(3) follows from (\ref{eq:3.17}) where $2p-\sinh2p=0$ for $p=0$
and $(2p-\sinh2p)^{\prime}=2-2\cosh2p<0$ for $p>0$. Hence $R_{1}(q_{t})$
does not admit any roots for $t>0$ in this case. \hfill$\square$\end{proof}

\begin{proposition}
\label{prop:3.5}Let $t>0$.
\begin{flalign}
(1)\qquad\mathrm{If}\,\lambda\in C_{1}\quad & \mathrm{then}\quad R_{2}(q_{t})=0\quad\mathrm{is\, impossible}.\label{eq:55-2}\\
(2)\qquad\mathrm{If}\,\lambda\in C_{2}\quad & \mathrm{then}\quad R_{2}(q_{t})=0,\quad\Longleftrightarrow\quad p=p_{2}^{n}(k)\textrm{ or }\mathrm{cn}\tau=0.\label{eq:56-1}\\
(3)\qquad\mathrm{If}\,\lambda\in C_{3}\quad & \mathrm{then}\quad R_{2}(q_{t})=0\quad\mathrm{is\, impossible}. & {}\label{eq:57-1}
\end{flalign}
\end{proposition}
\begin{proof}
Item (1) is given from (\ref{eq:3.6}) and Lemma \ref{lem:f2p}. Item
(2) is given from (\ref{eq:47}) and Proposition \ref{prop:f4p}.
Item (3) follows from (\ref{eq:3.18}) where $2p+\sinh2p=0$ for $p=0$
and $(2p+\sinh2p)^{\prime}=2+2\cosh2p>0$ for $p\geq0$. Hence $R_{2}(q_{t})$
does not admit any root for $t>0$ in this case. \hfill$\square$
\end{proof}

Let us now summarize the results obtained on the characterization
of the Maxwell strata.
\begin{theorem}
\label{thm:The-Maxwell-strata}The Maxwell strata $MAX^{i}\cap N^{j}$
are given as:

$(1)\qquad\mathrm{MAX}^{1}\cap N_{1}=\left\{ \nu\in N_{1}\quad\vert\quad p=p_{1}^{n}(k),\quad n\in\mathbb{N},\quad\mathrm{cn}\tau\neq0\right\} $,

$(2)\qquad\mathrm{MAX}^{1}\cap N_{2}=MAX^{1}\cap N_{3}=\emptyset$,

$(3)\qquad\mathrm{MAX}^{2}\cap N_{1}=MAX^{2}\cap N_{2}=\left\{ \nu\in N_{1}\cup N_{2}\quad\vert\quad p=2nK(k),\quad n\in\mathbb{N},\quad\mathrm{sn}\tau\neq0\right\} ,$

$(4)\qquad\mathrm{MAX}^{2}\cap N_{3}=\emptyset$,

$(5)\qquad\mathrm{MAX}^{6}\cap N_{1}=\mathrm{MAX}^{6}\cap N_{3}=\emptyset$,

$(6)\qquad\mathrm{MAX}^{6}\cap N_{2}=\left\{ \nu\in N_{2}\quad\vert\quad p=p_{2}^{n}(k),\quad n\in\mathbb{N},\quad\mathrm{cn}\tau\neq0\right\} $.\end{theorem}
\begin{proof}
This follows from the general description of the Maxwell strata and
Propositions \ref{prop:3.1}, \ref{prop:3.4} and \ref{prop:3.5}.
\hfill$\square$
\end{proof}

\subsection{\label{sub:Lim_Pts_Max_Set}Limit Points of the Maxwell Set}

It remains to consider the points at the boundary of the Maxwell strata
like the points in $N_{1}$ with $p=p_{1}^{n}(k),\quad\mathrm{cn}\tau=0$.
Since the action of reflections in the preimage of exponential map
is same for $\mathrm{SH}(2)$ and $\mathrm{SE}(2)$, it can be readily
seen using Proposition 5.8 \cite{max_sre} that when $\nu\in N_{1}$,
$p=p_{1}^{1}(k),\quad\mathrm{cn}\tau=0$ and when $\nu\in N_{2}$,
$p=p_{2}^{1}(k),\quad\mathrm{cn}\tau=0$ then $q_{t}=\mathrm{Exp}(\nu)$
is a conjugate point. The same reasoning applies to the case when
$\nu\in N_{1},\quad\mathrm{sn}\tau=0$ and $\nu\in N_{2},\quad\mathrm{sn}\tau=0$.
Thus we get the following statement.
\begin{proposition}
\label{prop:conj_pt_3.6}A point $q_{t}=\mathrm{Exp}(\nu)$ is conjugate
to the initial point $q_{0}$ if the following conditions hold:

$(1)\qquad\nu\in N_{1},\quad p=p_{1}^{n}(k),\quad n\in\mathbb{N},\quad\mathrm{cn}\tau=0$.

$(2)\qquad\nu\in N_{1}\cup N_{2},\quad p=2nK(k),\quad n\in\mathbb{N},\quad\mathrm{sn}\tau=0.$

$(3)\qquad\nu\in N_{2},\quad p=p_{2}^{n}(k),\quad n\in\mathbb{N},\quad\mathrm{cn}\tau=0.$
\end{proposition}

\subsection{Upper Bound on Cut Time}

It is well known that a normal extremal trajectory cannot be optimal
after the first Maxwell time. We now calculate the first Maxwell time
$t_{1}^{MAX}:C\rightarrow(0,+\infty]$.
\begin{proposition}
\label{prop:3.7}The first Maxwell time $t_{1}^{MAX}$ corresponding
to the reflections $\varepsilon^{1},\varepsilon^{2},\varepsilon^{6}$
is given as:
\begin{eqnarray*}
\lambda\in C_{1} & \implies & t_{1}^{MAX}(\lambda)=4K(k),\\
\lambda\in C_{2} & \implies & t_{1}^{MAX}(\lambda)=4kK(k),\\
\lambda\in C_{3}\cup C_{4}\cup C_{5} & \implies & t_{1}^{MAX}(\lambda)=+\infty.
\end{eqnarray*}
\end{proposition}
\begin{proof}
For $\lambda\in C_{1},\, C_{2},\, C_{3}$ apply Theorem \ref{thm:The-Maxwell-strata}
and Proposition \ref{prop:conj_pt_3.6}. For $\lambda\in C_{4}$ and
$\lambda\in C_{5}$, apply Theorems 5.4, 5.5 and Proposition 6.3 from
\cite{Extremal_Pseudo_Euclid}.\hfill$\square$
\end{proof}

Using Proposition \ref{prop:3.7} we get the following global upper
bound on the cut time $t_{cut}(\lambda)$ for extremal trajectories.
\begin{corollary}
For any $\lambda\in C$,
\end{corollary}
\begin{equation}
t_{\mathrm{cut}}(\lambda)\leq t_{1}^{MAX}(\lambda).\label{eq:3.34}
\end{equation}

\section{Conjugate Points}

In this section we study local optimality of sub-Riemannian geodesics
and compute the first conjugate time (i.e., the time of loss of local
optimality) along extremal trajectories. Let us recall certain important
facts related to conjugate points which will also outline the scheme
of further analysis. A point $q_{t}=\mathrm{Exp}(\lambda,t)$ is called
a conjugate point for $q_{0}$ if $\nu=(\lambda,t)=(\gamma,c,t)$
is a critical point of the exponential mapping, $q_{t}$ being its
critical value. In other words, this definition is given as: 
\[
d_{\nu}\mathrm{Exp}\,:\, T_{\nu}N\text{\ensuremath{\rightarrow}}T_{q_{t}}M\textrm{ is degenerate},
\]
where $d_{\nu}\mathrm{Exp}$ amounts to the Jacobian $J$ of the exponential
mapping i.e.,
\[
J=\frac{\partial(x_{t},y_{t},z_{t})}{\partial(\gamma,c,t)}=\left|\begin{array}{ccc}
\frac{\partial x_{t}}{\partial\gamma} & \frac{\partial x_{t}}{\partial c} & \frac{\partial x_{t}}{\partial t}\\
\frac{\partial y_{t}}{\partial\gamma} & \frac{\partial y_{t}}{\partial c} & \frac{\partial y_{t}}{\partial t}\\
\frac{\partial z_{t}}{\partial\gamma} & \frac{\partial z_{t}}{\partial c} & \frac{\partial z_{t}}{\partial t}
\end{array}\right|.
\]
According to the definition, roots of the equation $J=0$ give the
conjugate points and the time corresponding to these roots is called
the conjugate time. Carl Gustav Jacob Jacobi gave a geometric interpretation
of conjugate points according to which a conjugate point $q_{t}$
of a point $q_{0}$ is the point where an extremal trajectory meets
the envelope of the set of extremal trajectories through $q_{0}$
\cite{Bolza}. This is depicted in Figure \ref{fig:Conj_pts}. 
\begin{figure}
\centering{}\emph{\includegraphics[scale=0.6]{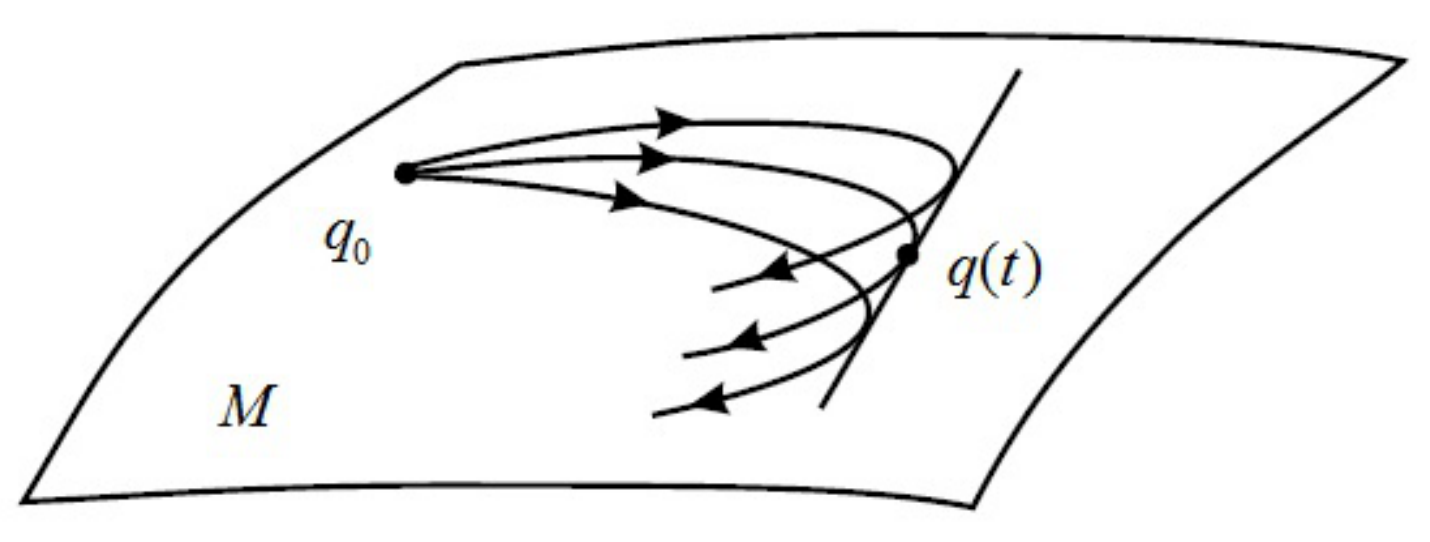}}\protect\caption{\label{fig:Conj_pts}Concept of conjugate point}
\end{figure}
In the local optimality analysis the first conjugate time is an important
notion as this is the time at which an extremal trajectory loses local
optimality. The first conjugate time is defined as:
\[
t_{1}^{\mathrm{conj}}(\lambda)=\textrm{inf}\left\{ t>0\quad\vert\quad t\textrm{ is a conjugate time along }\mathrm{Exp}(\lambda,s),\quad s\geq0\right\} .
\]

\subsection{Conjugate Points and Homotopy}

A lower bound of the form $t_{1}^{\mathrm{conj}}(\lambda)\geq t_{1}^{MAX}(\lambda)$
for all extremal trajectories $q(t)=\mathrm{Exp}(\lambda,t)$ was
proved in the Euler Elastic problem \cite{Euler_Conj}, sub-Riemannian
problem on SE(2) \cite{cut_sre1} and sub-Riemannian problem on the
Engel group \cite{engel_conj} via homotopy considering the fact that
the Maslov index (number of conjugate points along an extremal trajectory)
is invariant under homotopy \cite{Agrachev_Geo_OCP}. In order to
qualify for proof of absence of conjugate points below the lower bound
of the first conjugate time via homotopy, the optimal control problem
must satisfy a set of hypotheses \textbf{(H1})--(\textbf{H4)} \cite{Euler_Conj}
outlined below. 

Consider a general analytic optimal control problem on an analytic
manifold $M$:

\begin{eqnarray}
\dot{q} & = & f(q,u),\quad q\in M,\quad u\in U\subset\mathbb{R}^{m},\label{eq:4.1}\\
q(0) & = & q_{0},\quad q(t_{1})=q_{1},\quad t_{1}\textrm{ is fixed},\\
J & = & \intop_{0}^{t_{1}}\Phi(q(t),u(t))dt\rightarrow\min,\label{eq:4.3}
\end{eqnarray}
where $f(q,u)$ is a family of vector fields and $\Phi(q,u)$ is some
function on $M\times U$ analytic in system state $q\in M$ and control
parameter $u\in U$. Note that the sub-Riemannian problem on $M=\mathrm{SH}(2)$
(\ref{eq:2.1})--(\ref{eq:2.5}) is of this form. Let the control
dependent normal Hamiltonian of PMP for (\ref{eq:4.1})-(\ref{eq:4.3})
be given as:
\begin{equation}
h_{u}(\lambda)=\langle\lambda,f(q,u)\rangle-\Phi(q,u).
\end{equation}
Let a triple $\left(\tilde{u}(t),\lambda_{t},q(t)\right)$ represent
respectively the extremal control, extremal and extremal trajectory
corresponding to the normal Hamiltonian $h_{u}(\lambda)$. Let the
following hypotheses be satisfied for (\ref{eq:4.1})--(\ref{eq:4.3})
:

\textbf{(H1)} For all $\lambda\in T^{*}M$ and $u\in U$, the quadratic
form $\frac{\partial^{2}h_{u}}{\partial u^{2}}(\lambda)$ is negative
definite. This is the strong Legendre condition along the extremal
pair $(\tilde{u}(t),\lambda(t))$.

\textbf{(H2)} For any $\lambda\in T^{*}M$, the function $u\mapsto h_{u}(\lambda),\quad u\in U$,
has a maximum point $\bar{u}(\lambda)\in U$:
\[
h_{\bar{u}(\lambda)}(\lambda)=\max_{u\in U}h_{u}(\lambda),\quad\lambda\in T^{*}M.
\]

\textbf{(H3)} The extremal control $\tilde{u}(.)$ is a corank one
critical point of the endpoint mapping.

\textbf{(H4)} All trajectories of the Hamiltonian vector field $\overrightarrow{H}(\lambda),\quad H(\lambda)=\max_{u\in U}h_{u}(\lambda),\quad\lambda\in T^{*}M$,
are continued for $t\in[0,+\infty)$.

Under hypotheses (\textbf{H1})--(\textbf{H4}), the following is true
for the optimal control problem of the form (\ref{eq:4.1})--(\ref{eq:4.3}):
\begin{enumerate}
\item Normal extremal trajectories lose their local optimality (both strong
and weak) at the first conjugate point, see \cite{agrachev_sachkov}.
\item Along each normal extremal trajectory, conjugate times are isolated
one from another, see \cite{Euler_Conj}, \cite{cut_sre1}.
\end{enumerate}
We will apply the following statement for the proof of absence of
conjugate points via homotopy. 
\begin{proposition}
\label{prop:Coro2.2}\text{\emph{(Corollaries 2.2 and 2.3 \cite{Euler_Conj}).}}
Let $(u^{s}(t),\lambda_{t}^{s}),\, t\in[0,+\infty),\, s\in[0,1]$,
be continuous in parameter $s$ family of normal extremal pairs in
the optimal control problem \text{\emph{(\ref{eq:4.1})--(\ref{eq:4.3})}}
satisfying hypotheses \textbf{\emph{(H1)--(H4)}}. Let $s\mapsto t_{1}^{s}$
be a continuous function, $s\in[0,1]$, $t_{1}^{s}\in(0,+\infty)$.
Assume that for any $s\in[0,1]$ the instant $t=t_{1}^{s}$ is not
a conjugate time along the extremal $\lambda_{t}^{s}$ . If the extremal
trajectory $q^{0}(t)=\pi(\lambda_{t}^{0}),\, t\in(0,t_{1}^{0}],$
does not contain conjugate points, then the extremal trajectory $q^{1}(t)=\pi(\lambda_{t}^{1}),\, t\in(0,t_{1}^{1}]$,
also does not contain conjugate points.
\end{proposition}
It can be easily checked that the sub-Riemannian problem (\ref{eq:2.1})--(\ref{eq:2.5})
satisfies hypotheses (\textbf{H1})--(\textbf{H4}) and therefore Proposition
\ref{prop:Coro2.2} can be used to prove bounds of the first conjugate
time $t_{1}^{\textrm{conj}}$.

\subsection{Bounds on $t_{1}^{\textrm{conj}}$ for $\lambda\in C_{1}$}

Using the elliptic coordinates $(\varphi,k)$ defined in Section 5.3.1
\cite{Extremal_Pseudo_Euclid} and parametrization of extremal trajectories
(\ref{eq:2.8}), the Jacobian of the exponential mapping is given
as:
\begin{eqnarray}
J & = & \frac{\partial(x_{t},y_{t},z_{t})}{\partial(\varphi,k,t)}=\frac{J_{1}(p,\tau,k)}{(1-k^{2})^{2}(1-k\mathrm{sn}p\,\mathrm{sn}\tau)^{2}},\label{eq:4.5}\\
J_{1}(p,\tau,k) & = & -4k(\alpha_{1}+\alpha_{2}+\alpha_{3}),\label{eq:4.6}\\
\alpha_{1}(p,\tau,k) & = & \mathrm{sn}p\mathrm{\, cn}p\mathrm{\, dn}p\left(2\mathrm{E}(p)-p+k^{2}p\right),\nonumber \\
\alpha_{2}(p,\tau,k) & = & -\mathrm{dn}^{2}p\,\mathrm{sn}^{2}p-k^{2}\mathrm{sn}^{2}p\,\mathrm{cn}^{2}\tau,\nonumber \\
\alpha_{3}(p,\tau,k) & = & \mathrm{E}(p)\left(\mathrm{sn}^{2}p-\mathrm{sn}^{2}\tau\right)\left(\mathrm{E}(p)-p+k^{2}p\right),\nonumber 
\end{eqnarray}
where $p$ and $\tau$ for $\lambda\in C_{1}$ were defined in (\ref{eq:2.17}).
Plots of $J_{1}(p,\tau,k)$ are shown in Figures \ref{fig:-J1_5},
\ref{fig:-J1_6}.
\begin{figure}
\centering{}\includegraphics[scale=0.5]{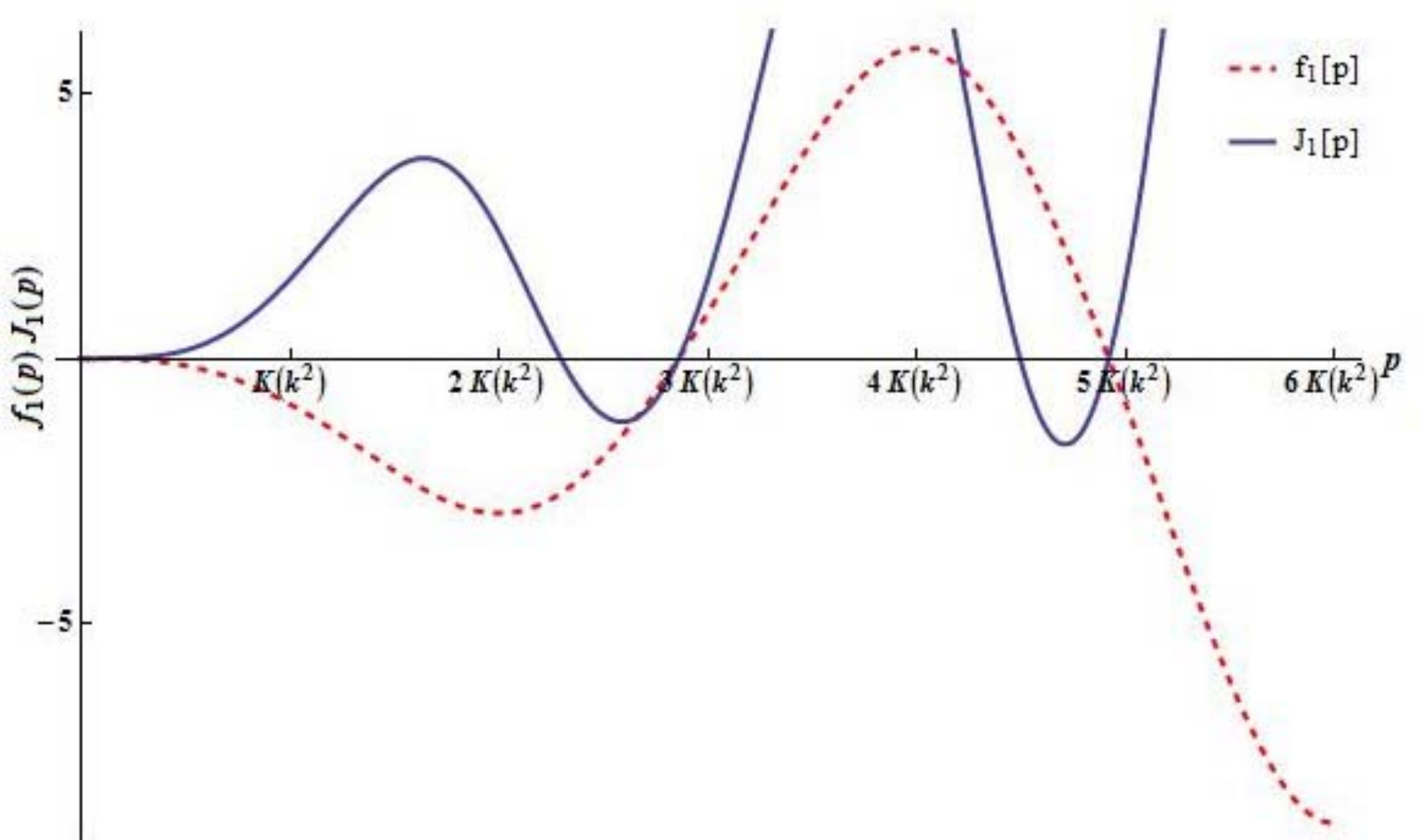}\protect\caption{\label{fig:-J1_5}$J_{1}(p,\tau,k)$ and $f_{1}(p)$ for $k$ = 0.5}
\end{figure}
\begin{figure}
\centering{}\includegraphics[scale=0.5]{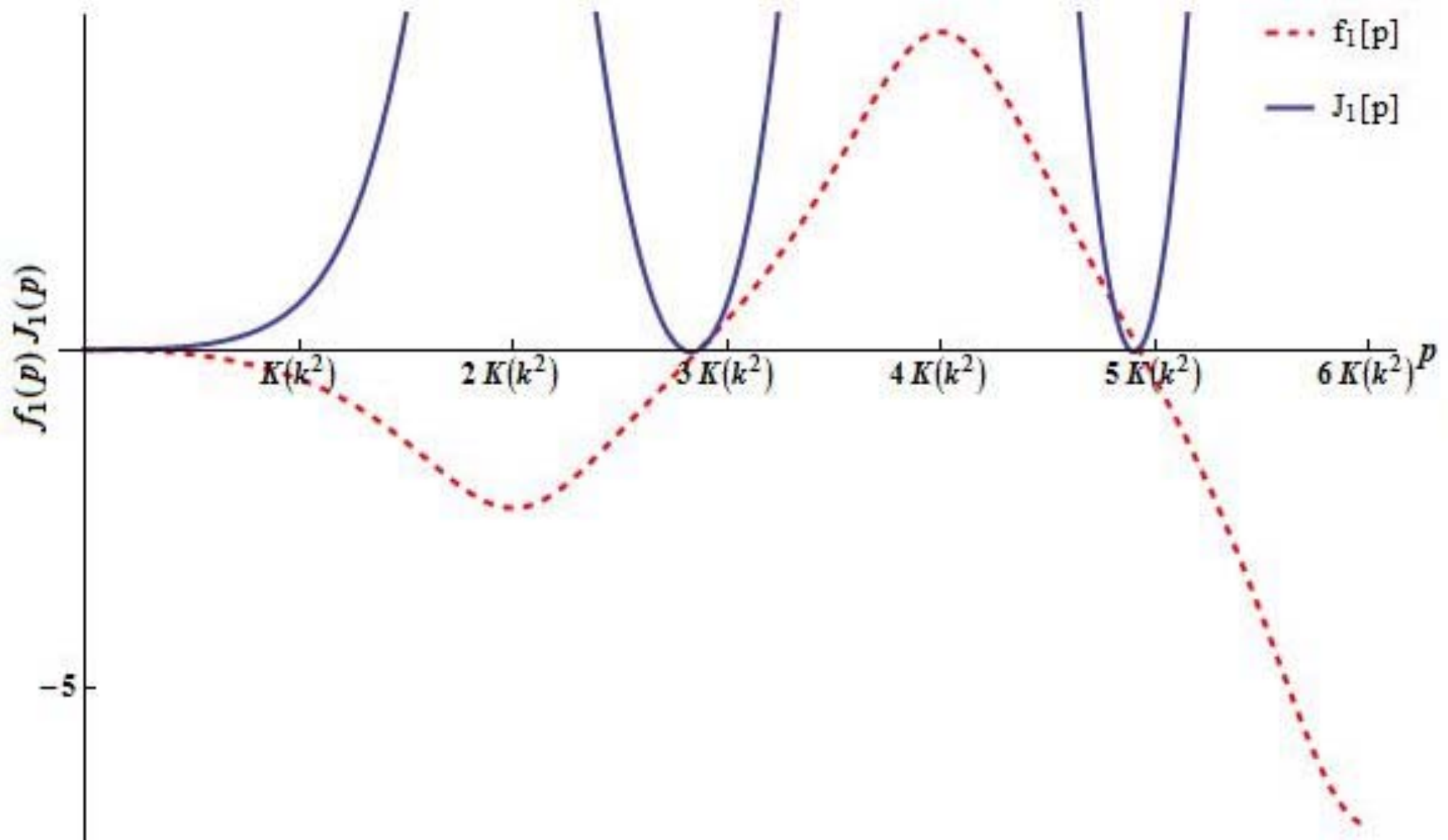}\protect\caption{\label{fig:-J1_6}$J_{1}(p,\tau,k)$ and $f_{1}(p)$ for $k$ = 0.9}
\end{figure}

\begin{lemma}
\label{lem:4.1}There exists $\widehat{k}\in(0,1)$ such that for
all $k\in(0,\widehat{k})$ and $p\in(0,\pi)$, the function $J_{1}$
is positive.\end{lemma}
\begin{proof}
The Taylor expansions of $J_{1}$ are given as:

\begin{eqnarray}
J_{1} & = & 4k\sin p(-p\cos p+\sin p),\qquad k\rightarrow0,\label{eq:4.7}\\
J_{1} & = & \frac{4}{3}kp^{4}+o(k^{2}+p^{2})^{4},\qquad k^{2}+p^{2}\rightarrow0.\label{eq:4.8}
\end{eqnarray}
From (\ref{eq:4.7}) it can be readily seen that in limit passage
of $k\rightarrow0^{+}$, $J_{1}>0$ for $p\in(0,\pi)$. Note that
$2K(0)=\pi$. Similarly from (\ref{eq:4.8}) it follows that $J_{1}>0$
when $k^{2}+p^{2}\rightarrow0^{+}$. \hfill$\square$\end{proof}

\begin{lemma}
\label{lem:4.4}If $k\in(0,1)$ and $p=2nK(k)$ for $n\in\mathbb{Z}$,
then $J_{1}\geq0$.\end{lemma}
\begin{proof}
Direct substitution of $p=2nK(k)$ to (\ref{eq:4.6}) gives:
\begin{equation}
J_{1}=16n^{2}kE(k)\left(E(k)-(1-k^{2})K(k)\right)\mathrm{sn}^{2}\tau.\label{eq:4.9}
\end{equation}
Since $f(k)=E(k)-(1-k^{2})K(k)>0$ because $f(0)=0$ and $f^{\prime}(k)=kK(k)>0$,
therefore, $J_{1}\geq0$. \hfill$\square$\end{proof}

\begin{lemma}
\label{lem:upper_bound}The system of equations
\begin{equation}
f_{1}(p,k)=0,\qquad J=0,\label{eq:4.10}
\end{equation}
is incompatible for $k\in(0,1)$, $p>0$.\end{lemma}
\begin{proof}
We denote
\[
E(u,k)=\int\limits _{0}^{u}\sqrt{1-k^{2}\sin^{2}\varphi}\, d\varphi,\quad F(u,k)=\int\limits _{0}^{u}\frac{d\varphi}{\sqrt{1-k^{2}\sin^{2}\varphi}},
\]
The system of equations (\ref{eq:4.10}), after the change $p=\mathrm{am}(u,k)$,
turns into: 
\begin{equation}
\begin{cases}
E(u,k)\cos u=\sqrt{1-k^{2}\sin^{2}u}\,\sin u,\\
F(u,k)\sqrt{1-k^{2}\sin^{2}u}\,\cos u=\sin u.
\end{cases}\label{eq:4.11}
\end{equation}
We prove that system (\ref{eq:4.11}) is incompatible for $k\in(0,1)$,
$u>0$. 

(1)$\quad$Let $0<u<\pi/2$. System (\ref{eq:4.11}) implies the equation:
\begin{align*}
\frac{E(u,k)}{\sqrt{1-k^{2}\sin^{2}u}} & =F(u,k)\sqrt{1-k^{2}\sin^{2}u},
\end{align*}
which is equivalent to the following equations:

\begin{eqnarray*}
 &  & \int\limits _{0}^{u}\frac{\sqrt{1-k^{2}\sin^{2}\varphi}}{\sqrt{1-k^{2}\sin^{2}u}}\, d\varphi=\int\limits _{0}^{u}\frac{\sqrt{1-k^{2}\sin^{2}u}}{\sqrt{1-k^{2}\sin^{2}\varphi}}\, d\varphi,\\
 &  & \int\limits _{0}^{u}\left(\frac{\sqrt{1-k^{2}\sin^{2}\varphi}}{\sqrt{1-k^{2}\sin^{2}u}}-\frac{\sqrt{1-k^{2}\sin^{2}u}}{\sqrt{1-k^{2}\sin^{2}\varphi}}\right)\, d\varphi=0,\\
 &  & \int\limits _{0}^{u}\frac{1-k^{2}\sin^{2}\varphi-(1-k^{2}\sin^{2}u)}{\sqrt{1-k^{2}\sin^{2}u}\sqrt{1-k^{2}\sin^{2}\varphi}}\, d\varphi=0,\\
 &  & \int\limits _{0}^{u}\frac{\sin^{2}u-\sin^{2}\varphi}{1-k^{2}\sin^{2}\varphi}\, d\varphi=0.
\end{eqnarray*}
The last equality is impossible since the function under the integral
is positive for $0<\pi<u$ (when $0<u<\pi/2$).

(2)$\quad$Equations of system (\ref{eq:4.11}) are violated when
$\cos u=0$ or $\sin u=0$, i.e., at the points $u=\frac{\pi k}{2}$,
$k\in\mathbb{N}$. This is checked immediately.

(3)$\quad$For $\frac{\pi}{2}<u<\pi$ system (\ref{eq:4.11}) is incompatible
since the function $\cos$ is negative, while the functions $\sin$,
$E$ and $F$ are positive.

(4)$\quad$It remains to consider the case $u>\pi$ for $\sin u\cos u\neq0$.
In this case we multiply the equations of the system, divide the first
equation by the second one, and get the following system: 
\[
\left\{ \begin{matrix} & \cos^{2}u\, E(u,k)F(u,k)=\sin^{2}u,\\
 & \frac{E(u,k)}{F(u,k)\sqrt{1-k^{2}\sin^{2}u}}=\sqrt{1-k^{2}\sin^{2}u}.
\end{matrix}\right.\quad\Leftrightarrow\quad\left\{ \begin{matrix} & E(u,k)F(u,k)=\tan^{2}u,\\
 & E(u,k)=F(u,k)(1-k^{2}\sin^{2}u).
\end{matrix}\right.
\]
The equality $1+\tan^{2}u=\cos^{-2}u$ and the equation $E(u,k)F(u,k)=\tan^{2}u$
imply: 
\[
\cos^{2}u=\frac{1}{1+E(u,k)F(u,k)}.
\]
 Since $1-k^{2}\sin^{2}u=1-k^{2}+k^{2}\cos^{2}u=1-k^{2}+\frac{k^{2}}{1+E(u,k)F(u,k)}$,
then the equation $E(u,k)=F(u,k)(1-k^{2}\sin^{2}u)$ is rewritten
as: 
\begin{equation}
E(u,k)=F(u,k)(1-k^{2})+\frac{k^{2}F(u,k)}{1+E(u,k)F(u,k)}.\label{eq:4.12}
\end{equation}
We have 
\begin{align*}
E(u,k)-(1-k^{2})F(u,k) & =\int\limits _{0}^{u}\left(\sqrt{1-k^{2}\sin^{2}\varphi}-\frac{1-k^{2}}{\sqrt{1-k^{2}\sin^{2}\varphi}}\right)\, d\varphi\\
 & =\int\limits _{0}^{u}\frac{1-k^{2}\sin^{2}\varphi-(1-k^{2})}{\sqrt{1-k^{2}\sin^{2}\varphi}}\, d\varphi=\int\limits _{0}^{u}\frac{k^{2}-k^{2}\sin^{2}\varphi}{\sqrt{1-k^{2}\sin^{2}\varphi}}\, d\varphi\\
 & =k^{2}\int\limits _{0}^{u}\frac{\cos^{2}\varphi}{\sqrt{1-k^{2}\sin^{2}\varphi}}\, d\varphi.
\end{align*}
Consequently, equation (\ref{eq:4.12}) takes the form: 
\[
k^{2}\int\limits _{0}^{u}\frac{\cos^{2}\varphi}{\sqrt{1-k^{2}\sin^{2}\varphi}}\, d\varphi=\frac{k^{2}F(u,k)}{1+E(u,k)F(u,k)},
\]
and after dividing both sides by $k^{2}$ we get: 
\[
\int\limits _{0}^{u}\frac{\cos^{2}\varphi}{\sqrt{1-k^{2}\sin^{2}\varphi}}\, d\varphi=\frac{F(u,k)}{1+E(u,k)F(u,k)}.
\]
Since $\frac{1}{\sqrt{1-k^{2}\sin^{2}\varphi}}>1$, $u>\pi$ then
there hold the inequalities: 
\[
\frac{\cos^{2}\varphi}{\sqrt{1-k^{2}\sin^{2}\varphi}}>\int\limits _{0}^{u}\cos^{2}\varphi\, d\varphi>\int\limits _{0}^{\pi}\cos^{2}\varphi\, d\varphi=\frac{\pi}{2}.
\]
Consequently, 
\begin{equation}
\frac{F(u,k)}{1+E(u,k)F(u,k)}=\int\limits _{0}^{u}\frac{\cos^{2}\varphi}{\sqrt{1-k^{2}\sin^{2}\varphi}}\, d\varphi>\frac{\pi}{2}.\label{eq:4.13}
\end{equation}
On the other hand, for $u\geq\pi/2$ we have $E(u,k)\geq E(k)>1$.
Consequently, 
\begin{equation}
\frac{F(u,k)}{1+E(u,k)F(u,k)}<\frac{F(u,k)}{1+F(u,k)}<1.\label{eq:4.14}
\end{equation}
Inequalities (\ref{eq:4.13}) and (\ref{eq:4.14}) contradict one
to another. This completes the proof of this lemma. \hfill$\square$\end{proof}

\begin{theorem}
\textcolor{red}{\label{thm:Conj_C1}}The first conjugate time for
$\lambda\in C_{1}$ is bounded as $4K(k)\leq t_{1}^{\mathrm{conj}}(\lambda)\leq2p_{1}^{1}(k)$.
Moreover, 
\begin{eqnarray*}
\lim_{k\rightarrow0^{+}}t_{1}^{\mathrm{conj}}(\lambda) & = & 2\pi,\\
\lim_{k\rightarrow1-0}t_{1}^{\mathrm{conj}}(\lambda) & = & +\infty.
\end{eqnarray*}
\end{theorem}
\begin{proof}
We first prove the lower bound of $t_{1}^{\mathrm{conj}}(\lambda)$.
We employ the approach adopted in the proof of Theorems 2.1, 2.2 \cite{cut_sre1}
and prove that for $\lambda\in C_{1}$ the interval $\left(0,2K(k)\right)$
does not contain conjugate points for the extremal trajectory $q(t)=\mathrm{Exp}(\lambda,t)$. 

Given any $\widehat{\lambda}\in C_{1}$, denote the corresponding
elliptic coordinates $(\widehat{\varphi},\widehat{k})$ and for $\widehat{t}=4K(\widehat{k})$
denote the corresponding parameters (\ref{eq:2.17}) as $\widehat{p}=\widehat{t}/2$
and $\widehat{\tau}=\widehat{\varphi}+\widehat{p}$. From the discussion
on conjugate points it is clear that for $p\in(0,\widehat{p})$, the
extremal trajectory $\widehat{q}(t)=\mathrm{Exp}(\widehat{\lambda},t)$
does not have conjugate points if $J_{1}\neq0$.

We choose the following family of curves in the plane $(k,p)$ continuous
in the parameter $s$:
\begin{equation}
\left\{ \left(k^{s},p^{s}\right)\quad\vert\quad s\in[0,1]\right\} ,\qquad k^{s}=s\widehat{k},\qquad p^{s}=2K(k^{s}).
\end{equation}
Clearly the endpoints of the curve $(k^{s},p^{s})$ are $(k^{0},p^{0})=(0,\pi)$
and $(k^{1},p^{1})=(\widehat{k},2K(\widehat{k}))$. The corresponding
family of extremal trajectories is given as:
\begin{eqnarray}
q^{s}(t) & = & \mathrm{Exp}(\varphi^{s},k^{s},t),\quad t\in[0,t^{s}],\quad s\in[0,1],\\
t^{s} & = & 2p^{s},\qquad\varphi^{s}=\widehat{\tau}-p^{s}.
\end{eqnarray}
From Lemma \ref{lem:4.1} it is clear that for sufficiently small
$s>0$, the Jacobian $J>0$ and hence the extremal trajectory $q^{s}(t)$
does not contain conjugate points for $p\in(0,2K(k^{s}))$, i.e.,
for $t\in(0,4K(k^{s}))$. Then from Proposition \ref{prop:Coro2.2}
it follows that the extremal trajectory $q^{s}(t)$ does not contain
conjugate points for any $s\in[0,1]$. Hence the extremal trajectory
$q(t)=\mathrm{Exp}(\lambda,t),\quad\lambda\in C_{1}$, does not contain
conjugate points in the interval $(0,4K(k))$ and therefore $t_{1}^{\mathrm{conj}}(\lambda)\geq4K(k)$. 

For proof of the upper bound apply Lemma \ref{lem:upper_bound}. Hence
it is proved that the first conjugate time is bounded as:

\begin{equation}
4K(k)\leq t_{1}^{\mathrm{conj}}(\lambda)\leq2p_{1}^{1}(k).
\end{equation}
From Lemma \ref{lem:4.1} and (\ref{eq:4.7}), the first root of $J$
occurs at $p=\pi$ and $\lim_{k\rightarrow0^{+}}2K(k)=\pi$. Therefore,
\[
\lim_{k\rightarrow0^{+}}t_{1}^{\mathrm{conj}}(\lambda)=4K(0)=2\pi.
\]
It can be readily seen that 
\[
\lim_{k\rightarrow1-0}t_{1}^{\mathrm{conj}}(\lambda)=+\infty.
\]
\hfill$\square$\end{proof}

\begin{remark}
For $\lambda\in C_{1}$, the instant $t=4K(k)$ is conjugate iff $\mathrm{sn}\tau=0$.
For proof substitute $n=1$ in (\ref{eq:4.9}) Lemma \ref{lem:4.4}
or alternatively substitute $\mathrm{sn}\tau=0$ in (\ref{eq:4.6}).
\end{remark}

\subsection{Bounds for $t_{1}^{\mathrm{conj}}(\lambda)$ for $\lambda\in C_{2}$}

Using the elliptic coordinates $(\psi,k)$ defined in Section 5.3.1
\cite{Extremal_Pseudo_Euclid} and the parametrization of extremal
trajectories (\ref{eq:2.8}), the Jacobian of the exponential mapping
is given as:
\begin{eqnarray}
J & = & \frac{\partial(x_{t},y_{t},z_{t})}{\partial(\psi,k,t)}=\frac{-kJ_{1}(p,\tau,k)}{(1-k^{2})^{2}(1-k\mathrm{sn}p\,\mathrm{sn}\tau)^{2}},\label{eq:4.19}
\end{eqnarray}
where $p$ and $\tau$ for $\lambda\in C_{2}$ were defined in (\ref{eq:2.18})
and $J_{1}$ is given by (\ref{eq:4.6}).
\begin{remark}
Notice that the Jacobian for $\lambda\in C_{2}$ (\ref{eq:4.19})
is just $(-k)$ times the expression of Jacobian for $\lambda\in C_{1}$
(\ref{eq:4.5}). Such a symmetry is unexpected and was not observed
in similar problems \cite{Euler_Conj}, \cite{cut_sre1}, \cite{engel_conj}.\end{remark}
\begin{theorem}
\textcolor{red}{\label{thm:Conj_C2}}The first conjugate time for
$\lambda\in C_{2}$ is bounded as $4kK(k)\leq t_{1}^{\mathrm{conj}}(\lambda)\leq2k\, p_{1}^{1}(k)$.
Moreover,
\begin{eqnarray*}
\lim_{k\rightarrow0}t_{1}^{\mathrm{conj}}(\lambda) & = & 0,\\
\lim_{k\rightarrow1-0}t_{1}^{\mathrm{conj}}(\lambda) & = & +\infty.
\end{eqnarray*}
\end{theorem}
\begin{proof}
\textcolor{black}{Since }$J=-kJ_{1}$\textcolor{black}{{} }for $\lambda\in C_{2}$,
therefore all arguments presented in the proof of Theorem \ref{thm:Conj_C1}
apply.\hfill$\square$\end{proof}

\begin{remark}
For $\lambda\in C_{2}$, the instant $t=4k\, K(k)$ is conjugate iff
$\mathrm{sn}\tau=0$. For proof substitute $n=1$ in (\ref{eq:4.9})
Lemma \ref{lem:4.4} or alternatively substitute $\mathrm{sn}\tau=0$
in (\ref{eq:4.6}).
\end{remark}

\subsection{Conjugate Points for the Cases of Critical Energy of Pendulum}
\begin{theorem}
\label{thm:Conj_C3}\text{\emph{(1)}}$\quad$If $\lambda\in C_{4}$,
then $t_{1}^{\mathrm{conj}}(\lambda)=2\pi$.

\text{\emph{(2)}}$\quad$If $\lambda\in C_{3}\cup C_{5}$, then $t_{1}^{\mathrm{conj}}(\lambda)=+\infty$. \end{theorem}
\begin{proof}
\text{(1)}$\quad$Let $\lambda\in C_{4}$. Take any continuous curve
$\lambda^{s}\in C,\quad s\in[0,1]$, such that $\lambda^{0}=\lambda$
and $\lambda^{s}\in C_{1}$ for $s\in(0,1]$. We have $\lim_{s\rightarrow0+}\lambda^{s}=\lambda$
and $\lim_{s\rightarrow0+}k^{s}=0$, thus $\lim_{s\rightarrow0+}t_{1}^{\mathrm{conj}}(\lambda^{s})=2\pi$
by Theorem \ref{thm:Conj_C1}. By continuity of the Jacobian $J(\lambda,t)=\frac{\partial q}{\partial(\lambda,t)}$,
we get $J(\lambda,2\pi)=\lim_{s\rightarrow0+}J\left(\lambda^{s},t_{1}^{\mathrm{conj}}(\lambda^{s})\right)=0$,
thus $2\pi$ is a conjugate time along the geodesic $\mathrm{Exp}(\lambda,t)$.
On the other hand, by Proposition \ref{prop:Coro2.2}, any interval
$(0,\tau]\subset(0,2\pi)$ does not contain conjugate times. Consequently,
$t_{1}^{\mathrm{conj}}(\lambda)=2\pi$. 

\text{(2)}$\quad$If $\lambda\in C_{3}\cup C_{5}$, we argue similarly.
By choosing continuous curve $\lambda^{s}\in C,\quad s\in[0,1]$,
such that $\lambda^{0}=\lambda$ and $\lambda^{s}\in C_{1}$ for $s\in(0,1]$.
Then $\lim_{s\rightarrow0+}\lambda^{s}=\lambda$ and $\lim_{s\rightarrow0+}k^{s}=1$,
thus $\lim_{s\rightarrow0+}t_{1}^{\mathrm{conj}}(\lambda^{s})=+\infty$
by Theorem \ref{thm:Conj_C1_nth}. Then we get $t_{1}^{\mathrm{conj}}(\lambda)=+\infty$
by Proposition \ref{prop:Coro2.2}.\hfill$\square$\end{proof}

\begin{theorem}
The two sided bounds on $t_{1}^{\mathrm{conj}}(\lambda)$ for $\lambda\in C_{1}$
given by \text{\emph{Theorem \ref{thm:Conj_C1}}} are exact in the
following sense:

\begin{flalign}
(1)\qquad\mathrm{If}\,\mathrm{sn}\tau=0\quad & \mathrm{then}\quad t_{1}^{\mathrm{conj}}(\lambda)=4K(k),\label{eq:4.21}\\
(2)\qquad\mathrm{If}\,\mathrm{cn}\tau=0\quad & \mathrm{then}\quad t_{1}^{\mathrm{conj}}(\lambda)=p_{1}^{1}(k). & {}\label{eq:4.22}
\end{flalign}
\end{theorem}
\begin{proof}
Substitute $\mathrm{sn}\tau=0$ for item (1) and $\mathrm{cn}\tau=0$
for item (2) in (\ref{eq:4.6}) respectively.\hfill$\square$\end{proof}

\begin{theorem}
The two sided bounds on $t_{1}^{\mathrm{conj}}(\lambda)$ for $\lambda\in C_{2}$
given by \text{\emph{Theorem \ref{thm:Conj_C2}}} are exact in the
following sense:

\begin{flalign}
(1)\qquad\mathrm{If}\,\mathrm{sn}\tau=0\quad & \mathrm{then}\quad t_{1}^{\mathrm{conj}}(\lambda)=4kK(k),\label{eq:4.23}\\
(2)\qquad\mathrm{If}\,\mathrm{cn}\tau=0\quad & \mathrm{then}\quad t_{1}^{\mathrm{conj}}(\lambda)=2kp_{1}^{1}(k). & {}\label{eq:4.24}
\end{flalign}
\end{theorem}
\begin{proof}
Substitute $\mathrm{sn}\tau=0$ for item (1) and $\mathrm{cn}\tau=0$
for item (2) in (\ref{eq:4.6}) respectively.\hfill$\square$
\end{proof}

\subsection{$n$-th Conjugate Time }

Computation of the first conjugate time is important in the study
of local optimality of the extremal trajectories. It turns out that
in the study of the sub-Riemannian wavefront, it is essential to bound
not only the first conjugate time, but all other conjugate times as
well. Hence in the following, we obtain the bounds for the $n$-th
conjugate time $t_{n}^{\mathrm{conj}}(\lambda)$ for $\lambda\in C_{1}\cup C_{2}$.
Recall that if $\lambda\in C_{3}\cup C_{5}$, then the trajectory
$\mathrm{Exp}(\lambda,t)$ is free of conjugate points (Theorem \ref{thm:Conj_C3})
and for $\lambda\in C_{4}$, the first conjugate time is given as
$t_{1}^{\mathrm{conj}}(\lambda)=2\pi$. Note that $\lambda\in C_{4}$
is the limiting case of $\lambda\in C_{1}$ with $\lim_{k\rightarrow0+}$
and therefore the bound on $n$-th conjugate time can be obtained
by taking $\lim_{k\rightarrow0+}t_{n}^{\mathrm{conj}}(\lambda),\quad\lambda\in C_{1}$. 
\begin{theorem}
\label{thm:Conj_C1_nth}The $n$-th conjugate time $t_{n}^{\mathrm{conj}}(\lambda)$
for $\lambda\in C_{1}$ is bounded as $4nK(k)\leq t_{2n-1}^{\mathrm{conj}}(\lambda)\leq2p_{1}^{n}(k)$
and $2p_{1}^{n}(k)\leq t_{2n}^{\mathrm{conj}}(\lambda)\leq4(n+1)K(k),\quad\forall n\in\mathbb{N}$.\end{theorem}
\begin{proof}
From Lemma \ref{lem:4.4} it is readily seen that $\forall p=2nK(k)$,
the expression of the Jacobian $J_{1}\geq0$. From Lemma \ref{lem:upper_bound},
$J_{1}\leq0$ at $p=p_{1}^{n}(k)$ i.e., the $n$-th root of the function
$f_{1}(p)$. Hence Jacobian $J$ (\ref{eq:4.5}) takes values of opposite
signs at the endpoints of the intervals $[2nK(k),p_{1}^{n}(k)]$ and
$[p_{1}^{n}(k),2(n+1)K(k)]$. Therefore, the $n$-th conjugate time
$t_{n}^{\mathrm{conj}}(\lambda)$is bounded as $4nK(k)\leq t_{2n-1}^{\mathrm{conj}}(\lambda)\leq2p_{1}^{n}(k)$
and $2p_{1}^{n}(k)\leq t_{2n}^{\mathrm{conj}}(\lambda)\leq4(n+1)K(k)\quad\forall n\in\mathbb{N}$.\hfill$\square$\end{proof}

\begin{corollary}
\label{cor:C2_conj_nth}From \text{\emph{Theorem \ref{thm:Conj_C2}}}
and \text{\emph{Theorem \ref{thm:Conj_C1_nth}}} we see that the $n$-th
conjugate time $t_{n}^{\mathrm{conj}}(\lambda)$ for $\lambda\in C_{2}$
is bounded as $4nkK(k)\leq t_{2n-1}^{\mathrm{conj}}(\lambda)\leq2kp_{1}^{n}(k)$
and $2kp_{1}^{n}(k)\leq t_{2n}^{\mathrm{conj}}(\lambda)\leq4(n+1)kK(k)$.\end{corollary}
\begin{theorem}
\label{thm:nth_Conj}The n-th conjugate times are bounded as: 
\begin{eqnarray*}
\lambda\in C_{1} & \implies & 4nK(k)\leq t_{2n-1}^{\mathrm{conj}}(\lambda)\leq2p_{1}^{n}(k),\quad2p_{1}^{n}(k)\leq t_{2n}^{\mathrm{conj}}(\lambda)\leq4(n+1)K(k),\\
\lambda\in C_{2} & \implies & 4nkK(k)\leq t_{2n-1}^{\mathrm{conj}}(\lambda)\leq2kp_{1}^{n}(k),\quad2kp_{1}^{n}(k)\leq t_{2n}^{\mathrm{conj}}(\lambda)\leq4(n+1)kK(k),\\
\lambda\in C_{4} & \implies & 2n\pi\leq t_{2n-1}^{\mathrm{conj}}(\lambda)\leq2p_{1}^{n}(0),\quad2p_{1}^{n}(0)\leq t_{2n}^{\mathrm{conj}}(\lambda)\leq2(n+1)\pi.
\end{eqnarray*}
\end{theorem}
\begin{proof}
The bounds follow from Theorem \ref{thm:Conj_C1_nth} and Corollary
\ref{cor:C2_conj_nth} for $\lambda\in C_{1}\cup C_{2}$. For $\lambda\in C_{4}$,
apply $\lim_{k\rightarrow0+}$ to bounds on $n$-th conjugate time
for $\lambda\in C_{1}$. \hfill$\square$\end{proof}

\begin{remark}
Notice that any extremal trajectory $q(t)$ either has countable number
of conjugate points, or is free of conjugate points. This alternative
is similar to that for LQ problems \cite{Agrachev_LQ}.
\end{remark}
It is conjectured that Rolle\textquoteright s theorem can be generalized
for sub-Riemannian problems such that between any two Maxwell points
there is one conjugate point, along any geodesic (the conjecture was
stated by A. A. Agrachev in a private conversation with the second
author).
\begin{proposition}
For any geodesic in the left-invariant sub-Riemannian problem on the
Lie group $\mathrm{SH}(2)$, between any two consecutive Maxwell points
there is exactly one conjugate point. \end{proposition}
\begin{proof}
By Theorem \ref{thm:nth_Conj}, the $n$-th conjugate time is bounded
as: 
\begin{eqnarray}
\lambda\in C_{1} & \implies & 4nK(k)\leq t_{2n-1}^{\mathrm{conj}}(\lambda)\leq2p_{1}^{n}(k),\quad2p_{1}^{n}(k)\leq t_{2n}^{\mathrm{conj}}(\lambda)\leq4(n+1)K(k),\label{eq:tconj_bound1}\\
\lambda\in C_{2} & \implies & 4nkK(k)\leq t_{2n-1}^{\mathrm{conj}}(\lambda)\leq2kp_{1}^{n}(k),\quad2kp_{1}^{n}(k)\leq t_{2n}^{\mathrm{conj}}(\lambda)\leq4(n+1)kK(k).\label{eq:tconj_bound2}
\end{eqnarray}
By Theorem \ref{thm:The-Maxwell-strata}, the $n{}^{th}$ Maxwell
time is bounded as:
\begin{eqnarray}
\lambda\in C_{1} & \implies & t_{2n-1}^{\mathrm{Max}}(\lambda)=4nK(k),\quad t_{2n}^{\mathrm{Max}}(\lambda)=2p_{1}^{n}(k),\label{eq:tmax_bound1}\\
\lambda\in C_{1} & \implies & t_{2n-1}^{\mathrm{Max}}(\lambda)=4nkK(k),\quad t_{2n}^{\mathrm{Max}}(\lambda)=2kp_{2}^{n}(k),\label{eq:tmax_bound2}
\end{eqnarray}
where $p_{1}^{n}(k)$ and $p_{2}^{n}(k)$ are the $n$-th roots of
the functions $f_{1}(p)=\mathrm{cn}p\,\mathrm{E}(p)-\mathrm{sn}p\,\mathrm{dn}p$
and $f_{4}(p)=-\mathrm{cn}p\,\mathrm{E}(p)+p\,\mathrm{cn}p(1-k^{2})+\mathrm{sn}p\,\mathrm{dn}p$
respectively and are both bounded as $(2nK(k),(2n+1)K(k))$ (\ref{eq:3.22}),
(\ref{eq:3.26}). By comparison of (\ref{eq:tconj_bound1}) with (\ref{eq:tmax_bound1})
and comparison of (\ref{eq:tconj_bound2}) with (\ref{eq:tmax_bound2}),
every conjugate point lies between consecutive Maxwell points for
$\lambda\in C_{1}\cup C_{2}$:
\[
t_{n}^{\mathrm{Max}}(\lambda)\leq t_{n}^{\mathrm{conj}}(\lambda)\leq t_{n+1}^{\mathrm{Max}}(\lambda),\quad n\in\mathbb{N}.
\]
From Theorems \ref{thm:The-Maxwell-strata} and \ref{thm:Conj_C3},
for $\lambda\in C_{3}\cup C_{5}$, there are neither any Maxwell points
nor any conjugate points whereas for $\lambda\in C_{4}$ there are
no Maxwell points. Hence, the proposition is pointless in these trivial
cases. \hfill$\square$
\end{proof}

\subsection{Sub-Riemannian Sphere and Wavefront}

Having explicit parametrization of the exponential mapping $\mathrm{Exp}(\lambda,t)$,
$\lambda\in C,\, t>0$ and the global bound on the cut time, we perform
a graphic study of some essential objects in the sub-Riemannian problem
on $\mathrm{SH}(2)$ in the rectifying coordinates $(R_{1},R_{2},z)$.
In particular we plot the sub-Riemannian sphere $S_{R}$ and the sub-Riemannian
wavefront $W_{R}$. Recall that the sub-Riemannian wavefront $W_{R}(q_{0};R)$
at $q_{0}$ is the set of end-points of geodesics with sub-Riemannian
length $R$ starting from $q_{0}$ and the sub-Riemannian sphere $S_{R}(q_{0};R)$
at $q_{0}$ is the set of end-points of minimizing geodesics of sub-Riemannian
length $R$ and starting from $q{}_{0}$:

\begin{eqnarray*}
W_{R} & = & \left\{ q=\mathrm{Exp}(\lambda,R)\in M\quad\vert\quad\lambda\in C\right\} ,\\
S_{R} & = & \left\{ q=\mathrm{Exp}(\lambda,R)\in M\quad\vert\quad\lambda\in C,\quad t_{\mathrm{cut}}(\lambda)\geq R\right\} =\left\{ q\in M\quad\vert\quad d(q_{0},q)=R\right\} ,
\end{eqnarray*}
where $R$ is the radius of sub-Riemannian sphere or wavefront and
$d\left(q_{0},q_{1}\right)=\mathrm{inf}\{l(q(.))\}$ is the sub-Riemannian
distance corresponding to sub-Riemannian length functional $l(q(.))$
(\ref{eq:2.4}) such that $q(.)$ is horizontal and $q(0)=q_{0},\quad q(t_{1})=q_{1}$.
Note the essential difference between sub-Riemannian wavefront and
sub-Riemannian sphere. The geodesics in sub-Riemannian wavefront are
only locally minimizing and drawn for time greater than the cut time
as well. On the contrary, the geodesics in sub-Riemannian sphere are
globally minimizing and therefore drawn for time not greater than
the upper bound of cut time and therefore, $S_{R}\subset W_{R}$,
but $S_{R}\neq W_{R}$ for $R>0$ and $S_{R}$ is the exterior component
of $W_{R}$ in the following sense:
\[
S_{R}=\partial(M\setminus W_{R}).
\]
A plot of the sub-Riemannian sphere is presented in Figure \ref{fig:SR_Sphere}
and plots of cutout of the sub-Riemannian wavefront are presented
in Figures \ref{fig:SR_Cutout}--\ref{fig:Sph_Planes_Cutout}. From
Figure \ref{fig:Sph_Planes_Cutout} it is clear that the wavefront
has self intersections in the surfaces $R_{i}(q_{t})=0$ and $z_{t}=0$
as expected from the general and complete description of Maxwell strata.
Figure \ref{fig:SR_Mat_WF} shows the Matryoshka of the sub-Riemannian
wavefront where self intersections in wavefronts of different radii
are clearly visible. In Figure \ref{fig:SR_Mat_Sph} we present the
Matryoshka of the sub-Riemannian spheres $S_{R}$ for different $R>0$.
Plots are presented from two different viewpoints for better visualization.
Note that as expected, exterior view of the sub-Riemannian sphere
is same as that of wavefront.
\begin{figure}
\centering{}\includegraphics[scale=0.5]{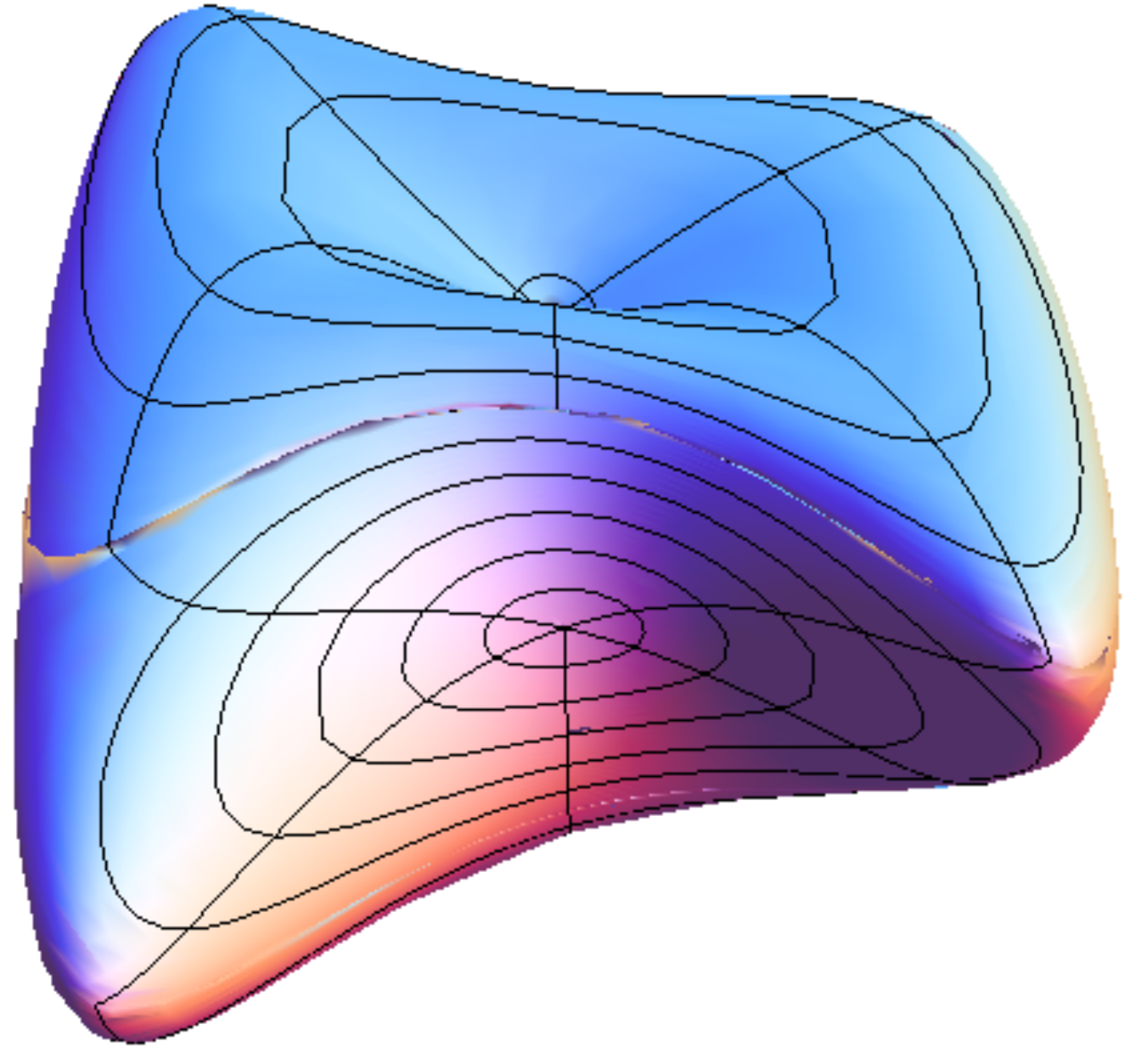}\protect\caption{\label{fig:SR_Sphere}Sub-Riemannian sphere for $R=2$}
\end{figure}
\begin{figure}
\centering{}\includegraphics[scale=0.4]{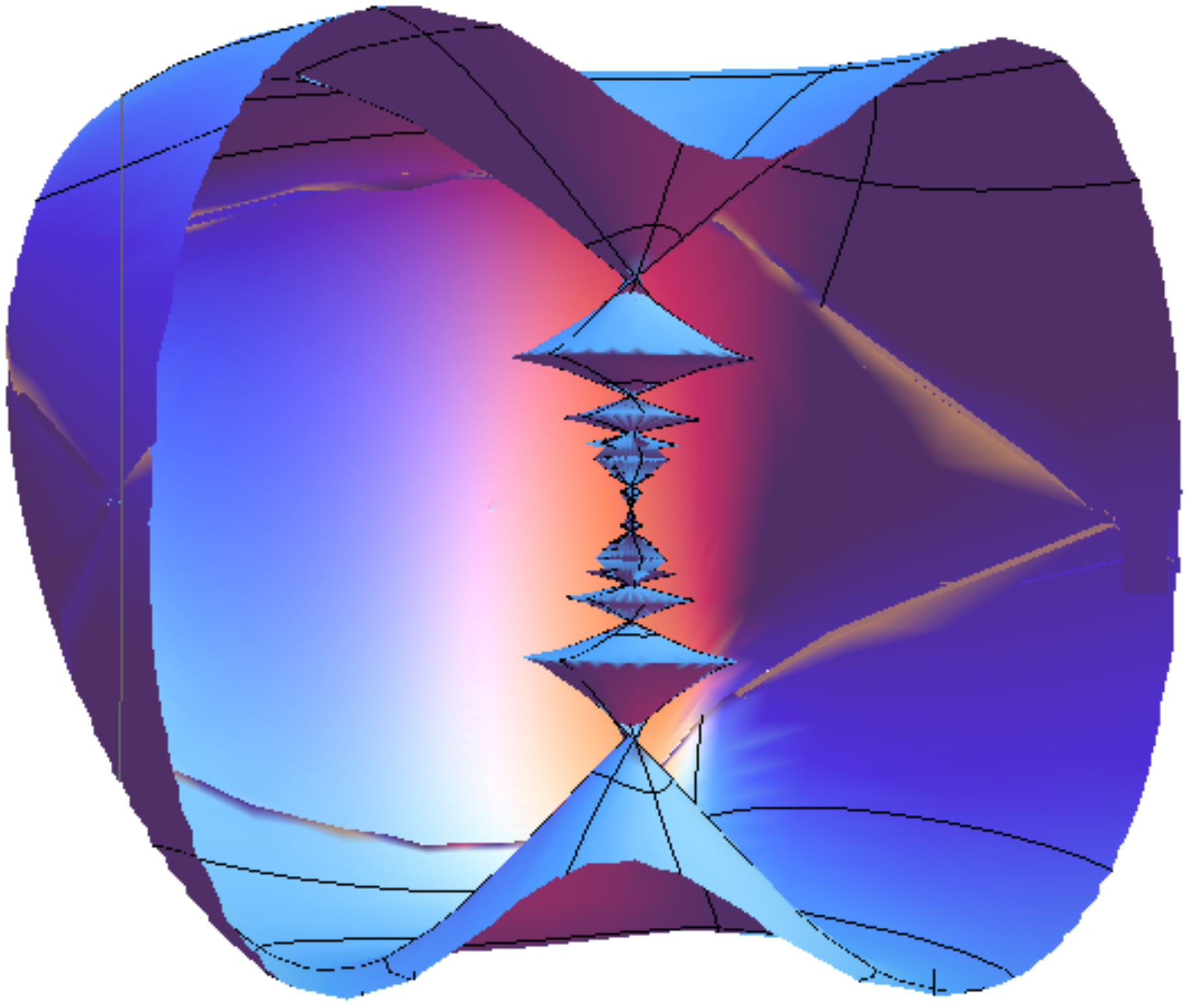}\protect\caption{\label{fig:SR_Cutout}Cutout of the sub-Riemannian wavefront for $R=2$}
\end{figure}

\begin{figure}
\begin{minipage}[t]{0.45\columnwidth}%
\begin{center}
\includegraphics[scale=0.3]{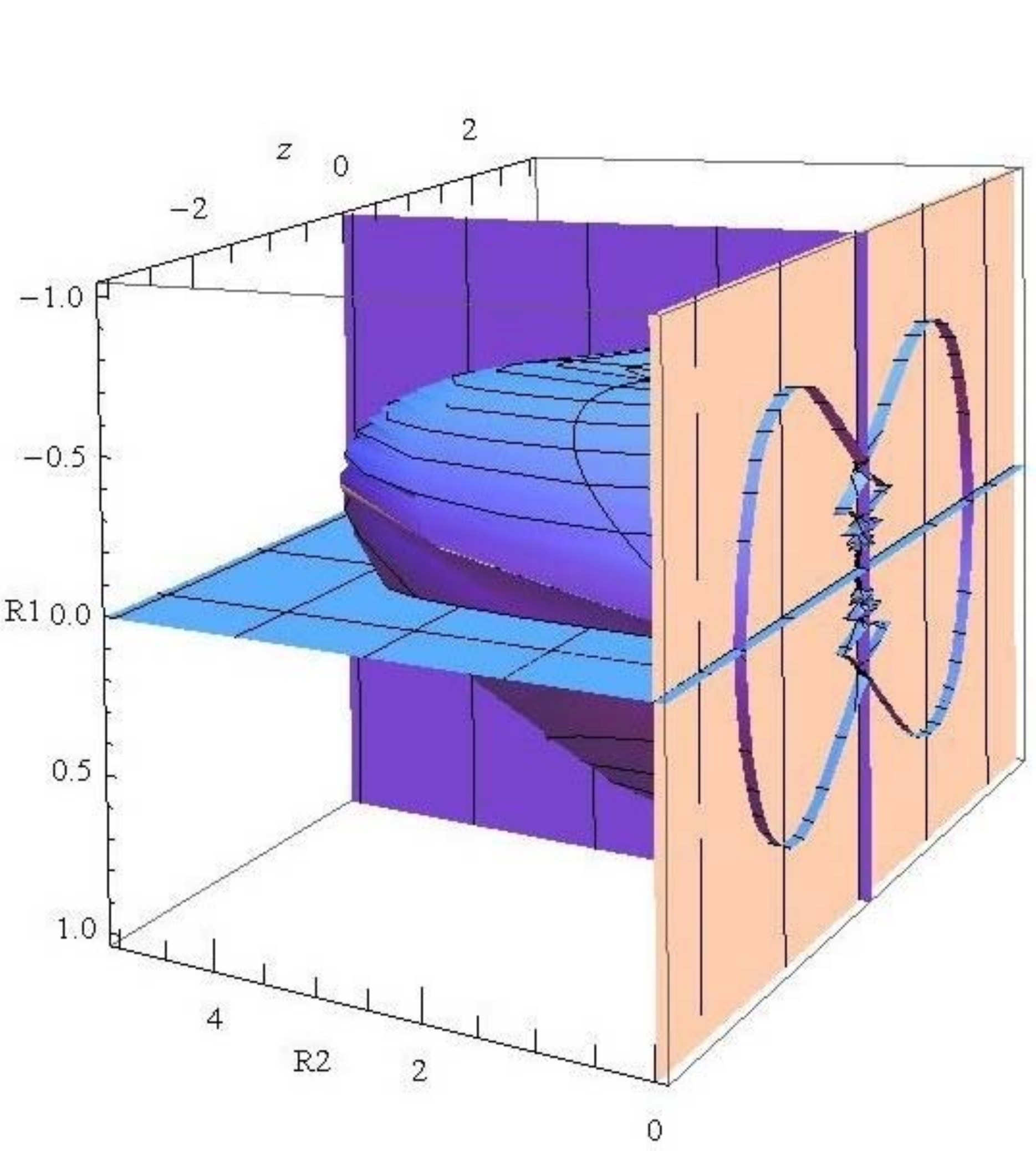}
\par\end{center}%
\end{minipage}%
\begin{minipage}[t]{0.45\columnwidth}%
\begin{center}
\includegraphics[scale=0.2]{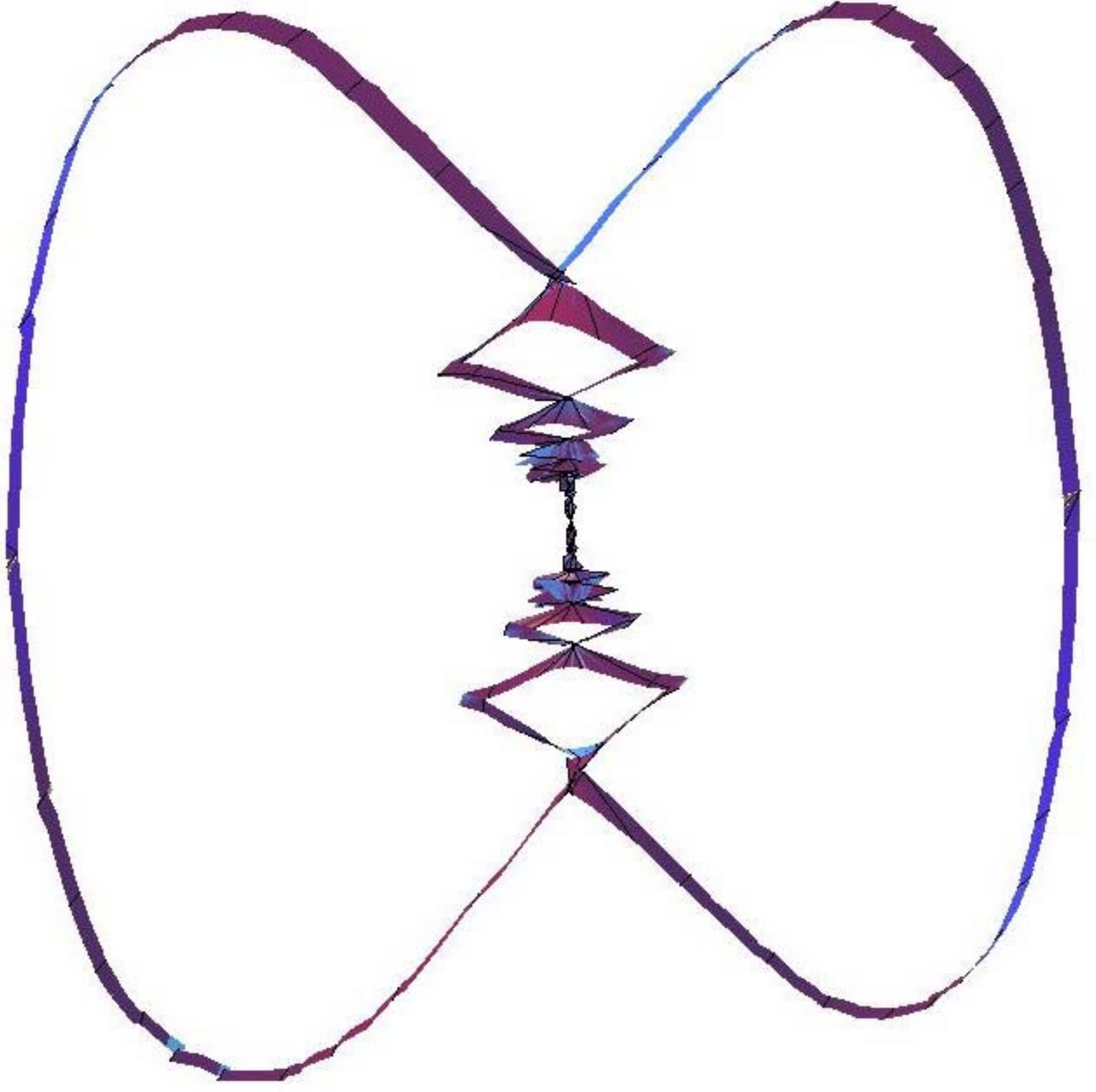}
\par\end{center}%
\end{minipage}\protect\caption{\label{fig:Sph_Planes_Cutout}Cutout of the sub-Riemannian wavefront
with self intersections in the planes $R_{i}(q_{t})=0$ and $z_{t}=0$
for $R=2$}
\end{figure}
\begin{figure}
\centering{}\includegraphics[scale=0.6]{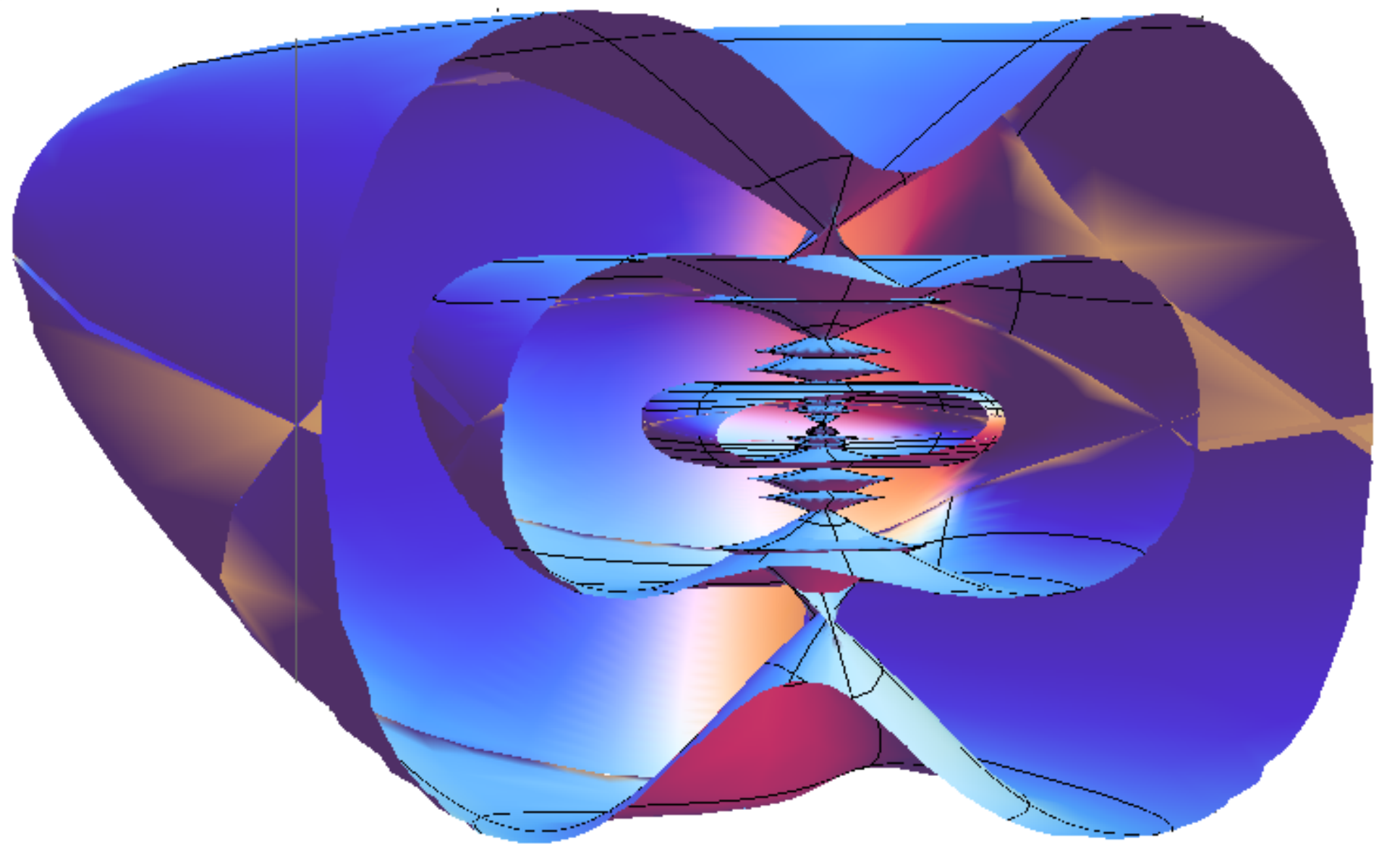}\protect\caption{\label{fig:SR_Mat_WF}Matryoshka of sub-Riemannian wavefronts $W_{R}$
for $R=1,2,3$}
\end{figure}
\begin{figure}
\begin{minipage}[t]{0.45\columnwidth}%
\begin{center}
\includegraphics[scale=0.3]{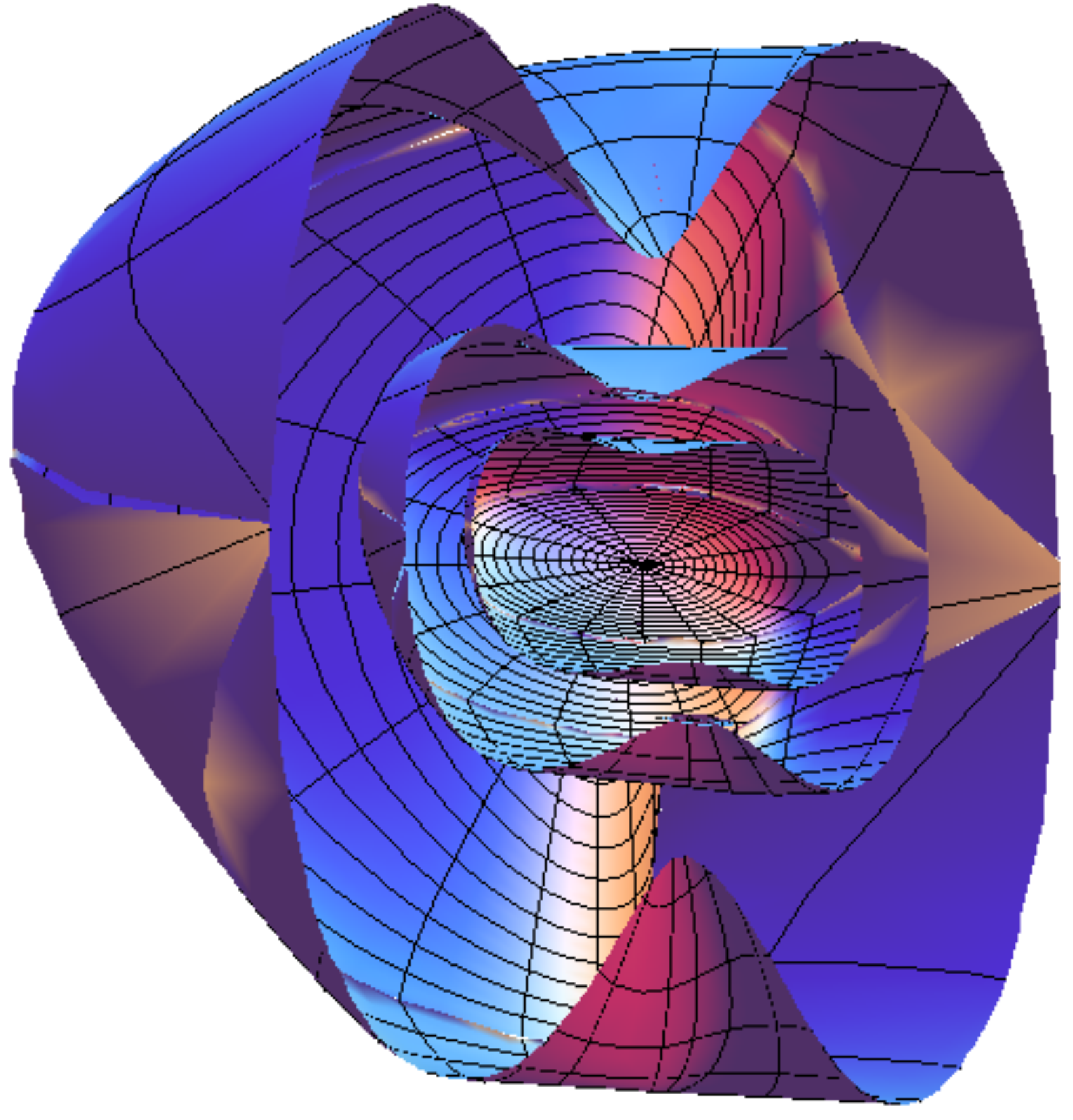}
\par\end{center}%
\end{minipage}%
\begin{minipage}[t]{0.45\columnwidth}%
\begin{center}
\includegraphics[scale=0.25]{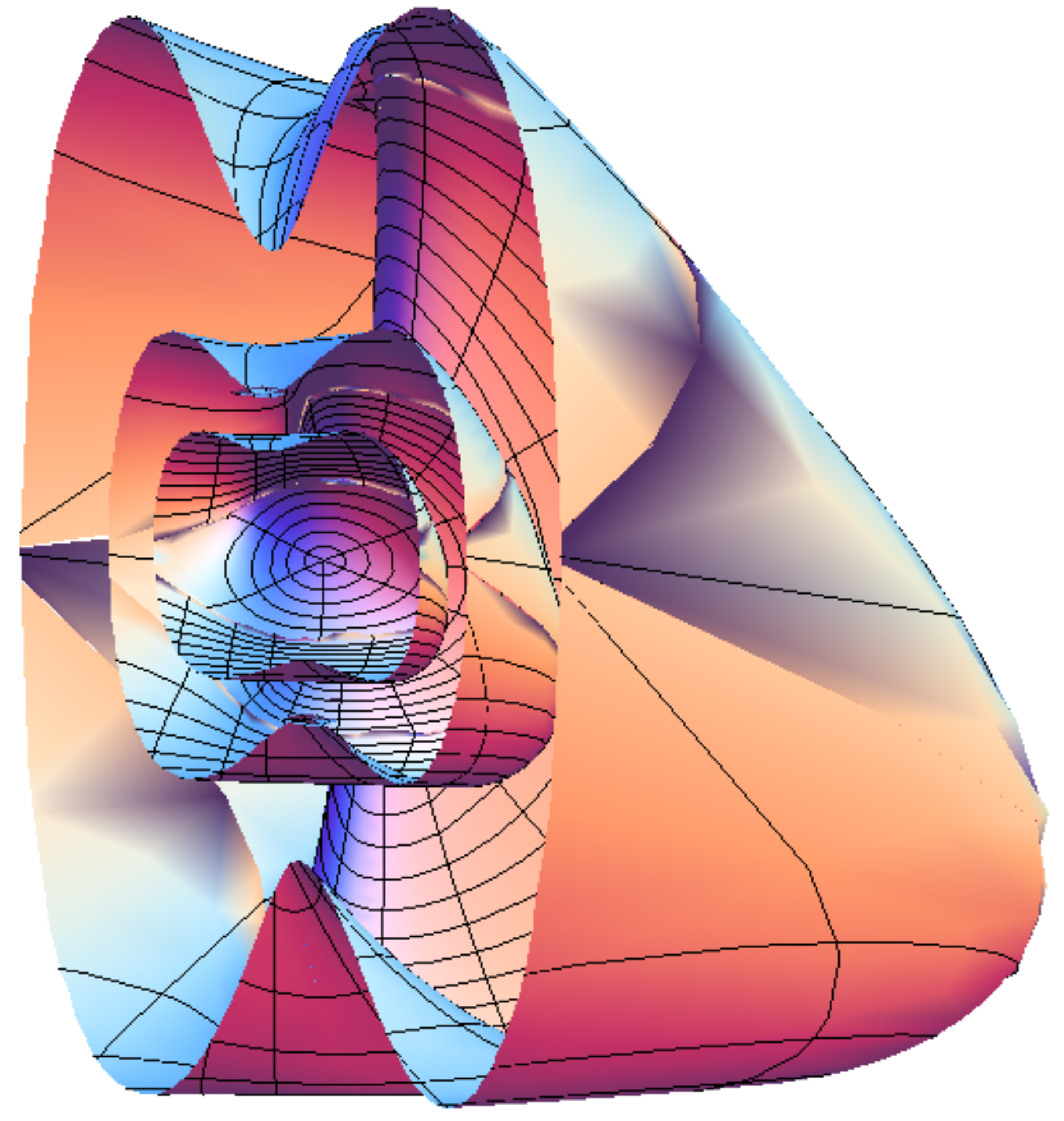}
\par\end{center}%
\end{minipage}\protect\caption{\label{fig:SR_Mat_Sph}Matryoshka of sub-Riemannian spheres $S_{R}$
for $R=1,2,3$}
\end{figure}

\section{Future Work}

In this paper we extended our research on the sub-Riemannian problem
on the Lie group $\mathrm{SH}(2)$ that was initiated in \cite{Extremal_Pseudo_Euclid}.
We obtained complete description of the Maxwell points, calculated
the upper bound on the cut time and computed the exact upper and lower
bounds for the $n$-th conjugate time, $n\in\mathbb{N}$. The next
research direction is the global optimality of sub-Riemannian geodesics.
In this regard we conjecture that the cut time is equal to the first
Maxwell time corresponding to the group of discrete symmetries of
the exponential mapping. This conjecture will be proved in our forthcoming
work on the sub-Riemannian problem on SH(2) \cite{cut_syn_sh2}.

\section{Conclusion}

The study of the sub-Riemannian problem on the group $\mathrm{SH}(2)$
is an important research goal that was initiated in \cite{Extremal_Pseudo_Euclid}
and has been continued in this work. We obtained a complete description
of the Maxwell points and global upper bound on the cut time. We also
computed the exact lower and upper bound of the $n$-th conjugate
time. We discovered an unexpected symmetry in the Jacobian expression
and the conjugate points in the case of oscillating and rotating pendulum
which hasn't been observed in optimality analysis in sub-Riemannian
problem on SE(2) \cite{cut_sre1}, the Engel group \cite{engel_conj}
and the Euler elastic problem \cite{Euler_Conj}. We conclude that
the $n$-th conjugate time is bounded by similar functions from below
and above for both $\lambda\in C_{1}$ and $\lambda\in C_{2}$. Moreover,
we showed that each geodesic contains either zero or a countable number
of conjugate points. We also proved a conjecture on generalized Rolle's
theorem for sub-Riemannian problem on Lie group $\mathrm{SH}(2)$.
\begin{acknowledgements}
We thank Prof. A. Yu. Popov for the proof of Lemma \ref{lem:upper_bound}.
\end{acknowledgements}

\bibliographystyle{unsrt}
\bibliography{ref}

\end{document}